\documentclass[a4paper,11pt]{amsart}

\let\oldtocsection=\tocsection
\let\oldtocsubsection=\tocsubsection
\let\oldtocsubsubsection=\tocsubsubsection

\usepackage[a4paper,left=4cm, right=4cm, top = 4cm, bottom =6cm]{geometry}

\renewcommand{\tocsection}[2]{\hspace{0em}\oldtocsection{#1}{#2}\bfseries}
\renewcommand{\tocsubsection}[2]{\hspace{1.8em}\oldtocsubsection{#1}{#2}}
\renewcommand{\tocsubsubsection}[2]{\hspace{4.4em}\oldtocsubsubsection{#1}{#2}}

\makeatletter
\renewcommand\subsection{\@startsection{subsection}{2}%
  \z@{-.5\linespacing\@plus-.7\linespacing}{.5\linespacing}%
  {\normalfont\scshape}}
\renewcommand\subsubsection{\@startsection{subsubsection}{3}%
  \z@{.5\linespacing\@plus.7\linespacing}{.5\linespacing}%
  {\normalfont\scshape}}
\makeatother

\usepackage[utf8]{inputenc}
\usepackage[T1]{fontenc}
\usepackage{tikz-cd}
\usepackage{paralist}
\usepackage{mathrsfs}
\usepackage{amssymb}
\usepackage{amsmath}
\usepackage{amsfonts}

\usepackage[framemethod=tikz]{mdframed}

\newtheorem{theorem}{Theorem}[section]

\newtheorem{lemma}[theorem]{Lemma}

\newtheorem{remark}[theorem]{Remark}
\newtheorem{Fact}[theorem]{Fact}
\newtheorem{Claim}[theorem]{Claim}

\theoremstyle{definition}
\newtheorem{definition}[theorem]{Definition}

\usepackage[pdftex]{hyperref}
\hypersetup{
    colorlinks=true, 
    linktoc=all,     
    linkcolor=blue,  
    allcolors=blue,
}

\newcommand{\dom}[1]{\ensuremath{\mathrm{dom}}(#1)}
\newcommand{\rng}[1]{\ensuremath{\mathrm{rng}}(#1)}

\newcommand{\set}[2]{\ensuremath{\{#1 \,|\, #2 \}}}
\newcommand{\seq}[2]{\ensuremath{\langle #1 \,|\, #2 \rangle}}
\newcommand{\restr}[2]{\ensuremath{#1 \! \upharpoonright \! #2}}

\newcommand{\Iff}{\Leftrightarrow}

\newcommand{\sub}{\subseteq}

\newcommand{\la}{\langle}
\newcommand{\ra}{\rangle}

\newcommand{\then}{\rightarrow}
\newcommand{\conc}{{}^{\smallfrown}}
\newcommand{\bb}{\mathbb}

\newcommand{\beq}{\begin{equation}}
\newcommand{\eeq}{\end{equation}}
\newcommand{\brm}{\begin{remark}\begin{rm}}
\newcommand{\erm}{\end{rm}\end{remark}}
\newcommand{\mx}{\mathrm}
\newcommand{\bce}{\begin{compactenum}}
\newcommand{\ece}{\end{compactenum}}

\newcommand{\cof}{\mathrm{cof}}
\newcommand{\cf}{\mathrm{cf}}

\newcommand{\Add}{\mathrm{Add}}

\newcommand{\Q}{\bb{Q}}
\renewcommand{\P}{\bb{P}}

\newcommand{\M}{\bb{M}}
\newcommand{\x}{\times}

\newcommand{\TP}{{\sf TP}}

\newcommand{\ZFC}{{\sf ZFC}}

\newcommand{\GCH}{{\sf GCH}}

\newcommand{\SR}{{\sf SR}}
\newcommand{\AP}{{\sf AP}}
\newcommand{\DSS}{{\sf DSS}}
\newcommand{\wKH}{{\sf wKH}}
\newcommand{\CSR}{\mathsf{CSR}}
\newcommand{\GMP}{\mathsf{GMP}}
\newcommand{\PFA}{\mathsf{PFA}}

\newcommand{\mf}{\mathfrak}
\newcommand{\UU}{\dot{U}}
\newcommand{\USupt}{\mathrm{USupt}}
\newcommand{\RSupt}{\mathrm{RSupt}}
\newcommand{\SSupt}{\mathrm{SSupt}}
\newcommand{\EVEN}{\mathrm{EVEN}}
\newcommand{\ODD}{\mathrm{ODD}}
\newcommand{\0}{\emptyset}
\newcommand{\uu}{\mathfrak{u}(\kappa)}

\newcommand{\lot}[1]{p^{\then #1}}

\DeclareMathOperator{\da}{\mathrel{\downarrow}}

\DeclareMathOperator{\bigger}{>}

\begin{document}

\title[Generalized cardinal invariants \ldots]{Generalized cardinal invariants for an inaccessible $\kappa$ with compactness at $\kappa^{++}$}

\author{Radek Honzik}
\address[Honzik]{
Charles University, Department of Logic,
Celetn{\' a} 20, Prague~1, 
116 42, Czech Republic
}
\email{radek.honzik@ff.cuni.cz}
\urladdr{logika.ff.cuni.cz/radek}

\author{{\v S}{\'a}rka Stejskalov{\'a}}
\address[Stejskalov{\'a}]{
Charles University, Department of Logic,
Celetn{\' a} 20, Prague~1, 
116 42, Czech Republic
}
\email{sarka.stejskalova@ff.cuni.cz}
\urladdr{logika.ff.cuni.cz/sarka}

\thanks{
  Both authors were supported by FWF/GA{\v C}R grant \emph{Compactness principles and combinatorics} (19-29633L), and R.\ Honzik  also by GA{\v C}R grant \emph{The role of set theory in modern mathematics} (24-12141S). To appear in \emph{Archive for Mathematical Logic}, version April 2025.}

\begin{abstract}
We study the relationship between non-trivial values of generalized cardinal invariants at an inaccessible cardinal $\kappa$ and compactness principles at $\kappa^+$ and $\kappa^{++}$. Let $\TP(\kappa^{++})$,  $\SR(\kappa^{++})$ and $\neg \wKH(\kappa^+)$ denote the tree property and stationary reflection on $\kappa^{++}$ and the negation of the weak Kurepa Hypothesis on $\kappa^+$, respectively.

We show that if the existence of a supercompact cardinal $\kappa$ with a weakly compact cardinal $\lambda$ above $\kappa$ is consistent, then the following are consistent as well (where $\mf{t}(\kappa)$ and $\uu$ are the tower number and the ultrafilter number, respectively): 

(i) There is an inaccessible cardinal $\kappa$ such that $\kappa^+ < \mf{t}(\kappa)= \mf{u}(\kappa)< 2^\kappa$ and $\SR(\kappa^{++})$ hold, and

(ii)  There is an inaccessible cardinal $\kappa$ such that $\kappa^+ = \mf{t}(\kappa) < \mf{u}(\kappa)< 2^\kappa$ and $\SR(\kappa^{++}), \TP(\kappa^{++})$ and $\neg \wKH(\kappa^+)$ hold.

The cardinals $\mf{u}(\kappa)$ and $2^\kappa$ can have any reasonable values in these models. We obtain these results by combining the forcing construction from \cite{F:u} due to Brooke-Taylor, Fischer, Friedman and Montoya with the Mitchell forcing and with (new and old) indestructibility results related to $\TP(\kappa^{++})$, $\SR(\kappa^{++})$ and $\neg \wKH(\kappa^+)$. Apart from $\mf{u}(\kappa)$ and $\mf{t}(\kappa)$ we also compute the values of $\mf{b}(\kappa)$, $\mf{d}(\kappa)$, $\mf{s}(\kappa)$, $\mf{r}(\kappa)$, $\mf{a}(\kappa)$, $\mx{cov}(\mathcal M_\kappa)$, $\mx{add}(\mathcal M_\kappa)$, $\mx{non}(\mathcal M_\kappa)$, $\mx{cof}(\mathcal M_\kappa)$ which will all be equal to $\mf{u}(\kappa)$.

In (ii), we compute $\mf{p}(\kappa) = \mf{t}(\kappa) = \kappa^+$ by observing that the $\kappa^+$-distributive quotient of the Mitchell forcing adds a tower of size $\kappa^+$.

Finally, as a corollary of the construction, we observe that items (i) and (ii) hold also for the traditional invariants on $\kappa = \omega$, using Mitchell forcing up to a weakly compact cardinal; in this case we also obtain the disjoint stationary sequence property $\DSS(\omega_2)$, which implies the negation of the approachability property $\neg \AP(\omega_2)$. 
\end{abstract}

	\maketitle

\section{Introduction}\label{sec:intro}

There has been an extensive research recently in the area of compactness principles at successor cardinals, and one of the questions is to what extent, if at all, these principles restrict the continuum function in the proximity of these cardinals. See for instance \cite{CK:ind}, \cite{arithmetic_paper}, or \cite{S:tp} for some examples. Extending this question, we can ask whether there are some restrictions for other cardinal invariants besides the continuum function. This has been done for instance in \cite{LH:Silver} and \cite{HS:u}, where the focus is on the cardinal invariants on singular strong limit cardinals. 

In this article we investigate cardinal invariants on an inaccessible $\kappa$ with the focus on the ultrafilter number $\mf{u}(\kappa)$ and the tower number $\mf{t}(\kappa)$. See Section \ref{sec:prelim} for the definitions of compactness principles $\SR(\kappa^{++})$ (stationary reflection), $\TP(\kappa^{++})$ (the tree property), $\neg \wKH(\kappa^{+})$ (the negation of the weak Kurepa Hypothesis), and $\DSS(\kappa^{++})$ (the disjoint stationary sequence property) we use in this article. We show in Theorems \ref{th:P_1} and \ref{th:P_2} that for any reasonable choice of cardinals $\uu$ and $2^\kappa$, $\kappa^+ \le \mf{t}(\kappa) = \mf{u}(\kappa) < 2^\kappa$ and $\kappa^+ = \mf{t}(\kappa) < \mf{u}(\kappa) < 2^\kappa$ are both consistent with $\SR(\kappa^{++})$, and $\kappa^+ = \mf{t}(\kappa) < \mf{u}(\kappa) < 2^\kappa$ is consistent with $\TP(\kappa^{++})$ and $\neg \wKH(\kappa^+)$. The consistency of $\kappa^+ < \mf{t}(\kappa) = \mf{u}(\kappa) < 2^\kappa$ with $\TP(\kappa^{++})$ and/or $\neg \wKH(\kappa^+)$ seems to be open at the moment (see Section \ref{sec:open} with open questions). In addition to $\uu$ and $\mf{t}(\kappa)$ we also compute the values of cardinal invariants $\mf{b}(\kappa), \mf{d}(\kappa), \mf{s}(\kappa), \mf{r}(\kappa)$ and $\mf{a}(\kappa)$, and also the invariants of the meager ideal $\mathcal M_\kappa$. See \cite{Bat:gen} for details regarding cardinal invariants at a regular uncountable $\kappa>\omega$. In Theorem \ref{th:P_3} we observe that our result also applies to the traditional cardinal invariants at $\omega$, where we in addition obtain the principle $\DSS(\omega_2)$, and hence $\neg \AP(\omega_2)$ (the negation of the approachability property). 

We use indestructibility results available for the above-mentioned compactness principles to argue that they hold in the final models: If $\kappa$ is regular with $\kappa^{<\kappa} = \kappa$ and $\lambda> \kappa$ is a large cardinal, it is known that suitable variants of Mitchell forcing $\M(\kappa,\lambda)$ force many compactness principles at $\lambda$ with $\lambda = \kappa^{++}$ in $V[\M(\kappa,\lambda)]$. Many of these compactness principles can be preserved in a further forcing extension via some $\P \in V[\M(\kappa,\lambda)]$, provided $\P$ has certain nice properties, such as being $\kappa^+$-cc. If all is set up correctly, the model $V[\M(\kappa,\lambda)*\dot{\P}]$ can satisfy both the compactness principles at $\kappa^{++}$ and some additional properties ensured by $\dot{\P}$.

In the present article, we show how to apply this strategy with the $\kappa^+$-Knaster and $\kappa$-directed closed forcing notion denoted $\P_\delta \da p_{\UU}$, introduced in Brooke-Taylor, Fischer, Friedman and Montoya \cite{F:u}. $\P_\delta \da p_{\UU}$ is a simplified version of a forcing from D{\v z}amonja and Shelah  \cite{DS:graph} which yields a model where $\kappa$ is inaccessible and $\uu < 2^\kappa$ (among other things). In both articles, $\kappa$ is assumed to be a supercompact cardinal and can retain its supercompactness in the final model. The ordinal $\delta$ in $\P_\delta \da p_{\UU}$ is the length of the iteration and has the prescribed cofinality which is equal to $\uu$ in the generic extension.

We show that it is possible to use $\P_\delta \da p_{\UU}$ in two different ways, obtaining two different forcing notions $\P_1$ and $\P_2$, and in effect two different patterns of cardinal invariants:

\begin{equation}
\P_1 := \M(\kappa,\lambda) * \dot{\P}_\delta \da p_{\UU},
\end{equation}
where $\dot{\P}_\delta \da p_{\UU}$ is defined in $V[\M(\kappa,\lambda)]$, and

\begin{equation}
\P_2 := \Add(\kappa,\lambda) * (\dot{\P}_\delta \da p_{\UU} \x \dot{R}),
\end{equation}
where we use that $\M(\kappa,\lambda)$ is equivalent to $\Add(\kappa,\lambda)*\dot{R}$ for some $\kappa^+$-distributive quotient $\dot{R}$,  and define $\dot{\P}_\delta \da p_{\UU}$ in the smaller model $V[\Add(\kappa,\lambda)]$.\footnote{$\Add(\kappa,\lambda)$ is the Cohen forcing for adding $\lambda$ many subsets of $\kappa$. For concreteness, we identify conditions in this forcing with partial function of size $<\kappa$ from $\lambda \x \kappa$ to $2$. If $I \sub \lambda$, $\Add(\kappa,I)$ has the natural meaning (the domain of functions is a subset of $I \x \kappa$).}

It is straightforward to observe that $\P_1$ forces stationary reflection at $\kappa^{++}$ by invoking an indestructibility result reviewed in Fact \ref{f:sr}, together with the pattern of  cardinal invariants computed in \cite{F:u} (see Theorem \ref{th:P_1}).

The forcing $\P_2$ forces in addition the tree property at $\kappa^{++}$ by Fact \ref{f:tp} and the negation of the weak Kurepa Hypothesis at $\kappa^+$ by Theorem \ref{th:wKH}, and in contrast to $\P_1$, it forces $\mf{p}(\kappa) = \mf{t}(\kappa) = \kappa^+$ due to the fact that $\dot{R}$ introduces a tower of size $\kappa^+$. See Theorem \ref{th:P_2} for details.

The article is structured as follows: in Section \ref{sec:prelim} we review basic notions and facts related to generalized cardinal invariants and compactness principles and prove an indestructibility Theorem \ref{th:wKH} for the negation of the weak Kurepa Hypothesis. In Section \ref{sec:review} we briefly review the forcing construction from \cite{F:u} in order to make the article (relatively) self-contained, and also to clarify some unclear points from \cite{F:u}. In Section \ref{sec:main} we prove main Theorems \ref{th:P_1} and \ref{th:P_2} for an inaccessible $\kappa$. In Section \ref{sec:open} we state some open questions and formulate a corollary of our construction which gives the consistency of an analogous configuration at $\omega$ using a forcing from \cite{BS:smallu}: $\omega_1 = \mf{t} < \mf{u} < 2^\omega$ plus  $\SR(\omega_2), \TP(\omega_2), \neg \wKH(\omega_1)$. In contrast to the situation for an inaccessible cardinal $\kappa$ -- where we do not know whether $\DSS(\kappa^{++})$ holds (see Remarks \ref{rm:new1} and \ref{rm:new2}), we are able to obtain $\DSS(\omega_2)$; see Theorem \ref{th:P_3} with forcing $\P_3$.

\brm \label{rm:new1}
There are other compactness principles for which a similar result can be attempted (see Question 4 in Section \ref{sec:open}). Let us provide a few comments on the negation of the approachability property $\neg \AP$ and the disjoint stationary sequence property $\DSS$ (see Definitions \ref{def:AP} and \ref{def:dss}).  Since $\DSS(\kappa^{++})$ is easy to preserve by  Theorem \ref{th:dss} and it implies $\neg \AP(\kappa^{++})$ by \cite[Corollary 3.7.]{KR:DSS}, it is natural to attempt to show the consistency of $\neg \AP(\kappa^{++})$ with cardinal invariants in Theorems \ref{th:P_1} and \ref{th:P_2} by means of $\DSS(\kappa^{++})$. However, it is not known whether $\DSS(\kappa^{++})$ holds in generic extensions by Mitchell forcings which force $\neg \AP(\kappa^{++})$ and do not add new subsets of $\omega$ (Mitchell forcing from Definition \ref{def:M} in this article is of this type hence our remark applies to it).  Note that the situation at $\kappa = \omega$ is different since it is known that Mitchell forcing $\M(\omega,\lambda)$ forces $\DSS(\omega_2)$ for a sufficiently large $\lambda$ (see \cite{KR:DSS}). See Theorem \ref{th:P_3} in this article. See also Remark \ref{rm:new2} for some more comments regarding $\DSS(\kappa^{++})$.
\erm

\section{Preliminaries}\label{sec:prelim}

\subsection{Generalized cardinal invariants}

We will review the definitions of the ultrafilter and tower numbers in this subsection because they are central to this article. Please consult \cite{F:u}, \cite{VFDMJSDS}, and \cite{Bat:gen} for other cardinal invariants which will appear in this article.

\begin{definition}\label{def:u}
Suppose $\kappa$ is an infinite cardinal. The ultrafilter number, $\uu$, is defined as follows:
$$\uu = \mx{min}\set{\mathcal B}{\mathcal B \mbox{ is a base of a uniform ultrafilter on }\kappa}.
$$
\end{definition}

Note for further reference that the ultrafilter does not need to be $\kappa$-complete even if $\kappa$ is measurable. Also note that $\kappa$ can be singular: see \cite{HS:u} for some results discussing $\uu$ and compactness at $\kappa^{++}$  for a singular $\kappa$.

Let now review the concept of a tower and of the tower number. For this definition we will assume $\kappa$ is regular. For $A, B \in [\kappa]^\kappa$, we write $T_\alpha \sub^* T_\beta$ if $|T_\alpha \setminus T_\beta| < \kappa$ (the \emph{almost-inclusion relation}).

\begin{definition}\label{def:tower}
Suppose $\kappa$ is a regular cardinal. We say that $\mathcal T \sub [\kappa]^\kappa$ is a \emph{tower} if $\mathcal T$ is reversely well-ordered by $\sub^*$,\footnote{$\mathcal T$ can be enumerated as an $\sub^*$-decreasing sequence $\seq{T_\alpha}{\alpha<\gamma}$ of elements in $[\kappa]^\kappa$, for some ordinal $\gamma$.} for every $X \sub \mathcal T$ of size $<\kappa$, $|\bigcap X| = \kappa$ (we say that $\mathcal T$ satisfies the \emph{strong intersection property, SIP}), and $\mathcal T$ has no pseudo-intersection, i.e.\ there is no $A \in [\kappa]^\kappa$ such that $A \sub^* T$ for all $T \in \mathcal T$.
\end{definition}

\begin{definition}\label{def:t}
Suppose $\kappa$ is a regular cardinal. The tower number, $\mf{t}(\kappa)$, is defined as follows:
$$\mf{t}(\kappa) = \mx{min}\set{\mathcal T}{\mathcal T \mbox{ is a tower on }\kappa}.
$$
\end{definition}

By a result of Malliaris and Shelah,  the tower number is the same as the pseudo-intersection number\footnote{Minimal size of a family $\mathcal F \sub [\kappa]^{\kappa}$ which has SIP but has no pseudo-intersection.} $\mf{p}(\kappa)$ if $\kappa = \omega$. For uncountable cardinals $\kappa$, it always holds $\mf{p}(\kappa) \le \mf{t}(\kappa)$, but the identity is open (see \cite{VFDMJSDS} for more details on this point). 

It is worth pointing out the following inequalities which are provable in $\ZFC$ (where $\mf{b}(\kappa)$ is the bounding number). 
$$\kappa^+ \le \mf{p}(\kappa) \le \mf{t}(\kappa) \le \mf{b}(\kappa) \le \uu.$$
See \cite{F:u} for proofs and more details. See also (\ref{eq1})--(\ref{eq3}) in this article for more inequalities for other cardinal invariants.

\subsection{Compactness principles}

Suppose $\lambda$ is regular and  $(T,<_T)$ is a tree of height $\lambda$. We say that $b \sub T$ is a \emph{cofinal branch} in $(T,<_T)$ if $(b,<_T)$ is linearly ordered and $b$ meets every level of $T$. We say that $(T,<_T)$ is a \emph{$\lambda$-tree} if it has height $\lambda$ and all levels of $T$ have size $<\lambda$. We usually write just $T$ for $(T,<_T)$ if there is no danger of confusion.

\begin{definition}
Suppose $\lambda$ is a regular cardinal. We say that the \emph{tree property} holds at $\lambda$, and we write $\TP(\lambda)$, if every $\lambda$-tree has a cofinal branch.
\end{definition}

\begin{definition}
Suppose $\lambda$ is a regular cardinal. We say that the \emph{negation of the weak Kurepa Hypothesis} holds at $\lambda$, and we write $\neg \wKH(\lambda)$, if there are no trees of height and size $\lambda$ which have at least $\lambda^+$-many cofinal branches.
\end{definition}

\brm
In this article, we will consider these principles at successor cardinals:  $\TP(\kappa^{++})$ and $\neg \wKH(\kappa^+)$ for a regular cardinal $\kappa$. Even though $\neg \wKH(\kappa^+)$ refers to trees of size $\kappa^+$ and $\TP(\kappa^{++})$ to $\kappa^{++}$-trees, they typically hold together (for instance in Mitchell-type forcings in this article which collapse a large cardinal $\lambda$ to become $\kappa^{++}$). The reason is that the large cardinal nature of the principle $\neg \wKH(\kappa^+)$ is related to the number of cofinal branches.
\erm

\begin{definition}
Suppose $\kappa$ is an infinite cardinal. We say that \emph{stationary reflection} holds at $\kappa^{++}$, and we write $\SR(\kappa^{++})$, if every stationary subset $S \sub \kappa^{++} \cap \cof(< \kappa^+)$ reflects at a point of cofinality $\kappa^+$; i.e.\ there is $\alpha < \kappa^{++}$ of cofinality $\kappa^+$ such that $\alpha \cap S$ is stationary in $\alpha$.
\end{definition}

Note that the set $\set{\alpha<\kappa^{++}}{\cf(\alpha) = \kappa^+}$ is always non-reflecting, hence the restrictions to cofinalities $<\kappa^+$ in $\SR(\kappa^{++})$ is essential. If $\lambda$ is inaccessible or a successor of a singlar cardinal, all stationary sets may reflect, but we will not consider such $\lambda$'s in this article.

As we discussed in Remark \ref{rm:new1}, the following two principles do not appear in Theorems \ref{th:P_1} and \ref{th:P_2}, but they appear in Theorem \ref{th:P_3}.

For cardinals $\kappa\le \lambda$, we denote by $\mathcal{P}_\kappa(\lambda)$ the set of all subsets of $\lambda$ of size $<\kappa$. The following property $\DSS$ was introduced in \cite{KR:DSS}:

\begin{definition}\label{def:dss}
Suppose $\kappa$ is an infinite cardinal. We say that $\kappa^{++}$ has the \emph{disjoint stationary sequence property}, $\DSS(\kappa^{++})$, if there are a stationary set $S \sub \kappa^{++} \cap \cof(\kappa^+)$ and a sequence $\seq{s_\alpha}{\alpha \in S}$ such that:
\begin{enumerate}[(i)]
\item For all $\alpha \in S$, $s_\alpha$ is a stationary subset of $\mathcal{P}_{\kappa^+}(\alpha)$;
\item For all $\alpha < \beta$ in $S$, $s_\alpha \cap s_\beta = \0$.
\end{enumerate}
\end{definition}

For a regular cardinal $\lambda$ and sequence $\bar{a}=\seq{a_\alpha}{\alpha<\lambda}$ of bounded subsets of $\lambda$, we say that an ordinal $\gamma<\lambda$ is approachable with respect to $\bar{a}$ if there is an unbounded subset $A\sub\gamma$ of order type $\cf(\gamma)$ and for  all $\beta<\gamma$ there is $\alpha<\gamma$ such that $A\cap\beta=a_\alpha$.

Let us define the \emph{ideal $I[\lambda]$ of approachable subsets of $\lambda$}:

\begin{definition}
$S\in I[\lambda]$ if and only if there are a sequence $\bar{a}= \seq{a_\alpha}{\alpha<\lambda}$ of bounded subsets of $\lambda$ and a club $C\sub \lambda$ such that every $\gamma\in S\cap C$ is approachable with respect to $\bar{a}$.
\end{definition}

\begin{definition}\label{def:AP}
We say that \emph{the approachability property holds} at $\lambda$ if $\lambda \in I[\lambda]$ (or equivalently, there is a club subset of $\lambda$ in $I[\lambda]$), and we write $\AP(\lambda)$.
\end{definition}

In \cite[Corollary 3.7.]{KR:DSS} Krueger showed that $\DSS(\kappa^{++})$ implies $\neg \AP(\kappa^{++})$.

\brm \label{rm:new2}
At the moment, the only known method for obtaining the principle $\DSS(\kappa^{++})$ for an uncountable $\kappa$ is to use a variant of the Mitchell forcing which adds $\kappa^{++}$ many new subsets of $\omega$ (see \cite[Theorem 9.1]{KR:DSS}). This modification is not suitable for us because we need to preserve $\kappa^{<\kappa} = \kappa$ in Theorems \ref{th:P_1} and \ref{th:P_2}.  The difficulty of the problem comes in part from fact that it is related to another open question dealing with possible generalizations of Gitik's theorem that adding a real makes $({\mathcal P}_\kappa\lambda)^V$ (for relevant $\kappa,\lambda$) co-stationary in generic extensions which contain the real (see \cite{Gitik:real} and comments in \cite[Section 7]{KR:DSS}).
\erm

\subsection{Indestructibility of some compactness principles}

A strong form of indestructibility is known for stationary reflection. It works over any model and requires just an appropriate chain condition (we formulate it for a double successor to fit our present purposes):

\begin{Fact}[Indestructibility of stationary reflection, \cite{HS:u}] \label{f:sr}
Suppose $\SR(\kappa^{++})$ holds and $\Q$ is $\kappa^+$-cc and preserves $\kappa$. Then $\Q$ forces $\SR(\kappa^{++})$.
\end{Fact}

Known results for preservation of the tree property by forcings with an appropriate chain condition are less general and are usually formulated for variants of Mitchell-type forcings.\footnote{In a recent article \cite{HLS:gmp}, Honzik, Lambie-Hanson and Stejskalova proved an indestructibility result which holds for all models of a certain extension of $\ZFC$: they showed that Cohen forcing for adding any number of Cohen subsets of $\kappa$, with $\kappa = \kappa^{<\kappa}$, preserves $\TP(\kappa^{++})$ over any model of the guessing model principle $\GMP_{\kappa^{++}}$. In particular $\TP(\omega_2)$ is preserved by Cohen forcing at $\omega$ over any model of $\PFA$. }  Let us define a specific version of Mitchell forcing which we will use in the present article. The set $\mathcal A \sub \lambda$ in (\ref{A}) is a parameter of the construction and determines at which stages collapsing occurs.\footnote{Any choice of $\mathcal A$ which ensures conditions (\ref{eq:M1}), (\ref{eq:M2}) and is cofinal in $\lambda$ yields $\TP(\kappa^{++})$ and $\neg \wKH(\kappa^+)$ in $V[\M^{\mathcal A}(\kappa,\lambda)]$. For $\neg \AP(\kappa^{++})$, $\mathcal A$ must be chosen more carefully to ensure the approximation property of the quotient. Definition (\ref{A}) ensures $\neg \AP(\kappa^{++})$; see \cite{8fold} for more details.} For the purposes of this article we will use the following:
\beq \label{A} \mathcal A = \set{\alpha < \lambda}{\alpha \mbox{ is inaccessible, but not a limit of inaccessibles}}.\eeq

\begin{definition}[Mitchell forcing]\label{def:M}
Suppose $\kappa = \kappa^{<\kappa}$,  $\lambda> \kappa$ is an inaccessible limit of inaccessibles and $\mathcal A$ is as in (\ref{A}). Conditions in $\M^\mathcal A(\kappa,\lambda)$ are pairs $(p^0,p^1)$ such that $p^0$ is a condition in the Cohen forcing for adding $\lambda$ many subsets of $\kappa$, i.e.\ $p^0 \in \Add(\kappa,\lambda)$, and $p^1$ is a function with domain $\dom{p^1} \sub \mathcal A$ of size at most $\kappa$. For $\alpha$ in the domain of $p^1$, $p^1(\alpha)$ is an $\Add(\kappa,\alpha)$-name and \begin{equation} 1_{\Add(\kappa,\alpha)} \Vdash p^1(\alpha) \in \Add(\kappa^+,1)^{V[\Add(\kappa,\alpha)]}.\end{equation} The ordering is defined as follows: $(p^0,p^1) \le (q^0,q^1)$ iff $p^0 \le q^0$ in $\Add(\kappa,\lambda)$ and the domain of $p^1$ extends the domain of $q^1$, and for every $\alpha \in \dom{q^1}$, \begin{equation} \restr{p^0}{\alpha} \Vdash p^1(\alpha) \le q^1(\alpha),\end{equation} where $\restr{p^0}{\alpha}$ is the natural restriction of $p^0$ to $\alpha \x \kappa$. We will abuse notation and write $\M(\kappa,\lambda)$ instead of $\M^{\mathcal A}(\kappa,\lambda)$ if there is no danger of confusion.
\end{definition}

$\M(\kappa,\lambda)$ collapses cardinals in the open interval $(\kappa^+,\lambda)$, and forces $\lambda = \kappa^{++}$ and $2^\kappa = \lambda$. It also forces various compactness principles depending on the largeness of $\lambda$. We will use the following product and quotient analysis due to Abraham \cite{ABR:tree}. 

For every inaccessible limit of inaccessibles $\alpha < \lambda$ or $\alpha = 0$, there is a projection $\pi_\alpha$ in $V[\M(\kappa,\alpha)]$ as follows:
\begin{equation}\label{eq:M1}
\pi_\alpha: \Add(\kappa,[\alpha,\lambda)) \x R^1_\alpha \to_{\mx{onto}}\M(\kappa,[\alpha,\lambda)),
\end{equation}
where $\Add(\kappa,[\alpha,\lambda))$ is the $\kappa$-Cohen forcing on coordinates $[\alpha,\lambda)$, $\M(\kappa,[\alpha,\lambda))$ is the Mitchell forcing defined at the interval $[\alpha,\lambda)$ and $R^1_\alpha$ is a term forcing which is $\kappa^+$-closed in $V[\M(\kappa,\alpha)]$. We write $R^1$ for $R^1_0$.

For every inaccessible limit of inaccessibles $\alpha < \lambda$ or $\alpha = 0$, the forcing $\M(\kappa, \lambda)$ can be written as
\begin{equation}\label{eq:M2}
\M(\kappa,\lambda) \equiv \M(\kappa,\alpha) * \dot{\M}(\kappa,[\alpha,\lambda)) \equiv \M(\kappa,\alpha) * (\Add(\kappa, [\alpha,\lambda)) * \dot{R}_\alpha),
\end{equation}
where $\dot{R}_\alpha$ is forced to be $\kappa^+$-distributive by $\M(\kappa,\alpha) * \Add(\kappa,[\alpha,\lambda))$ and $\equiv$ denotes forcing equivalence. We write $\dot{R}$ for $\dot{R}_0$.

\begin{Fact}\label{f:M}
Mitchell forcing $\M^{\mathcal A}(\kappa,\lambda) = \M(\kappa,\lambda)$ with $\mathcal A$ as in (\ref{A}) forces the principles $\SR(\kappa^{++})$, $\TP(\kappa^{++})$, $\neg \wKH(\kappa^+)$ and $\neg \AP(\kappa^{++})$ whenever $\kappa^{<\kappa} = \kappa$ and $\lambda>\kappa$ is weakly compact. If $\kappa = \omega$, then $\M(\omega,\lambda)$ in addition forces $\DSS(\omega_2)$ as well.
\end{Fact}

\begin{proof}
See \cite{8fold} for the proofs of $\SR(\kappa^{++})$, $\TP(\kappa^{++})$ and $\neg \AP(\kappa^{++})$. The argument for $\neg \wKH(\kappa^+)$ is a variant of the tree property argument (see Theorem \ref{th:wKH} in the present article for more details). See \cite{KR:DSS} for the proof of $\DSS(\omega_2)$ and more information on the principle $\DSS$ in general.
\end{proof}

\brm
We have defined $\M(\kappa,\lambda)$ for $\lambda$ which is an inaccessible limit of inaccessible cardinals to ensure that $\mathcal A$ is cofinal in $\lambda$. This is sufficient for lifting-type arguments because $\lambda$ is always an inaccessible limit of inaccessibles in a target model of a weakly compact embedding with critical point $\lambda$, at which cardinal (\ref{eq:M1}) and (\ref{eq:M2}) are applied (in fact, the analysis in (\ref{eq:M1}) and (\ref{eq:M2}) holds at every $\alpha < \lambda$ but this is not relevant for us so we have not mentioned it above). With a different choice of the parameter $\mathcal A$, $\lambda$ may be just an inaccessible cardinal. Finally note that for $\alpha \in \mathcal A$, $\M(\kappa,\lambda)$ can be written as \beq \label{embed} \M(\kappa,\alpha) * (\dot{\Add}(\kappa^+,1) \x \Add(\kappa,1)) * \dot{\M}(\kappa,[\alpha+1,\lambda)),\eeq and for $\alpha \not \in \mathcal A$ as \beq\M(\kappa,\alpha) *  \Add(\kappa,[\alpha,\alpha')) * (\dot{\Add}(\kappa^+,1) \x \Add(\kappa,1)) *
 \dot{\M}(\kappa,[\alpha'+1,\lambda)),\eeq where $\alpha'>\alpha$ is the least element of $\mathcal A$ above $\alpha$. Thus at an inaccessible limit of inaccessibles $\alpha <\lambda$, the tail of the forcing after $\M(\kappa,\alpha)$ has the $\kappa^+$-approximation property (see \cite{8fold}), which is essential for showing $\neg \AP(\kappa^{++})$ in $V[\M(\kappa,\lambda)]$.
\erm 

The following is an indestructibility result for the tree property we will use in this article:

\begin{Fact}[Indestructibility of the tree property, \cite{HS:ind}]\label{f:tp}
Suppose $\kappa = \kappa^{<\kappa}$ and $\lambda > \kappa$ is weakly compact. Suppose $\Add(\kappa,\lambda)$ forces that $\dot{\Q}$ is a forcing notion which is $\kappa^+$-cc and preserves $\kappa$. Then $\M(\kappa,\lambda) * \dot{\Q}$ forces $\TP(\kappa^{++})$.
\end{Fact}

\begin{proof}
The proof of Theorem 3.2 in \cite{HS:ind} is formulated for $\mathcal A = \lambda$. However, it is easy to check that the argument does not depend on the exact definition of $\mathcal A$: whenever $\M^{\mathcal A}(\kappa,\lambda)$ satisfies  conditions (\ref{eq:M1}) and (\ref{eq:M2}) and forces $2^\kappa = \lambda = \kappa^{++}$, the indestructibility argument of Theorem 3.2 applies. See the proof of Theorem \ref{th:wKH} for indestructibility of $\neg \wKH(\kappa^+)$ in this article for more details.
\end{proof}

Let us now extend Fact \ref{f:tp} to the negation of the weak Kurepa Hypothesis. 

\begin{definition}\label{def:sq}
Suppose $\kappa$ is a regular cardinal. We say that a forcing $\P$ is \emph{productively $\kappa$-cc} iff $\P \x \P$ is $\kappa$-cc.
\end{definition}

Notice that every $\kappa$-Knaster forcing is productively $\kappa$-cc, and every productively $\kappa$-cc forcing is $\kappa$-cc.

The following is well known and will be useful for us:

\begin{Fact}\label{f:sq}
Suppose $\kappa$ is a regular cardinal and $T$ is a tree of height $\kappa$. If $\Q$ is productively $\kappa$-cc, then $\Q$ does not add new cofinal branches to $T$.
\end{Fact}

The productive chain condition is not so well-behaved with regard to preservation under iterations as the regular chain condition or the Knaster condition, but there is a weaker characterization. To formulate it, let us introduce the following notation: Suppose $\dot{\Q}$ is a $\P$-name. We can view it artificially as a $\P \x \P$-name by modifying it to depend only on the first coordinate or the second coordinate of $\P \x \P$, obtaining $\dot{\Q}^1$ and $\dot{\Q}^2$, respectively.\footnote{\label{ft:mod} Define recursively a function ${}^*$ from $\P$-names to $\P \x \P$-names so that $\sigma^* = \set{[(p,1),\tau^*]}{(p,\tau) \in \sigma}$ for $\dot{\Q}^1$ and $\set{[(1,p),\tau^*]}{(p,\tau) \in \sigma}$ for $\dot{\Q}^2$.}

Recall the following characterization which holds for the regular chain condition: 

\begin{equation} \label{eq:c} \P * \dot{\Q} \mbox{ is }\kappa\mbox{-cc} \Iff \P \mbox{ is }\kappa\mbox{-cc and } \P \Vdash \dot{\Q} \mbox{ is } \kappa \mbox{-cc}.\end{equation}

For the productively $\kappa$-cc condition, we have the following weaker characterization:

\begin{lemma}\label{lm:sq}
Let $\P * \dot{\Q}$ be a forcing notion and $\kappa$ a regular cardinal. Then the following hold:
\begin{enumerate}[(i)]
\item If $\P*\dot{\Q}$ is productively $\kappa$-cc, then $\P$ is productively $\kappa$-cc and forces that $\dot{\Q}$ is productively $\kappa$-cc.
\item If $\P$ is productively $\kappa$-cc and $\P \x \P$ forces that $\dot{\Q}^1 \x \dot{\Q}^2$ is $\kappa$-cc, then $\P* \dot{\Q}$ is productively $\kappa$-cc.
\item Suppose $\P$ is productively $\kappa$-cc and $\P$ forces that $\dot{\Q}$ is $\kappa$-Knaster. Then $\P*\dot{\Q}$ is productively $\kappa$-cc.
\end{enumerate}
\end{lemma}

\begin{proof}
For (i),  $\P$ is productively $\kappa$-cc because there is a natural regular embedding from $\P \x \P$ into $(\P * \dot{\Q}) \x (\P * \dot{\Q})$. We use (\ref{eq:c}) repeatedly for the second claim: If $\P * \dot{\Q}$ is productively $\kappa$-cc, then $\P * \dot{\Q} \Vdash \P *\dot{\Q}$ is $\kappa$-cc. This is equivalent to $\P \Vdash \dot{\Q} \Vdash \P *\dot{\Q}$ is $\kappa$-cc, which is in turn equivalent to $\P \Vdash \dot{\Q} * \P * \dot{\Q}$ is $\kappa$-cc, which readily implies $\P \Vdash \dot{\Q} \x \dot{\Q}$ is $\kappa$-cc.\footnote{We can continue to obtain a partial converse to the implication in (ii): $\P \Vdash \dot{\Q} * \P * \dot{\Q}$ is $\kappa$-cc implies $\P \Vdash \P \Vdash \dot{\Q} \x \dot{\Q}$ is $\kappa$-cc, equivalently $\P \x \P \Vdash \dot{\Q}^1 \x \dot{\Q}^1$ is $\kappa$-cc, and by mutual genericity of generic filters for $\P \x \P$, $\P \x \P \Vdash \dot{\Q}^2 \x \dot{\Q}^2$ is $\kappa$-cc. But this is still weaker than the condition which implies $\P * \dot{\Q}$ is productively $\kappa$-cc in (ii).}

For (ii),  assume for contradiction that $A = \set{((p_\alpha,\dot{q}_\alpha),(p'_\alpha,\dot{q}'_\alpha))}{\alpha<\kappa}$ is an antichain in $(\P * \dot{\Q}) \x (\P*\dot{\Q})$ of size $\kappa$. We shall need the following observation: if $((p_\alpha,\dot{q}_\alpha),(p'_\alpha,\dot{q}'_\alpha))$ and $((p_\beta,\dot{q}_\beta),(p'_\beta,\dot{q}'_\beta))$, $\alpha < \beta$, are in $A$, and there are conditions $p,p'$ in $\P$ such that $p \le p_\alpha,p_\beta$ and $p' \le p'_\alpha,p'_\beta$, then either $p \Vdash \dot{q}_\alpha \perp \dot{q}_\beta$ or $p' \Vdash \dot{q}'_\alpha \perp \dot{q}'_\beta$, which can be reformulated for the forcing $(\P \x \P) * (\dot{\Q}^1 \x \dot{\Q}^2)$ as: \begin{equation}\label{eq:sq}(p,p') \Vdash (\dot{q}_\alpha,\dot{q}'_\alpha) \perp (\dot{q}_\beta,\dot{q}'_\beta),\end{equation} by replacing the names $\dot{q}_\alpha, \dot{q}'_\alpha$, etc., as explained in Footnote \ref{ft:mod}.

Define $A^* = \set{((p_\alpha,p'_\alpha),\mathrm{pair}(\dot{q}_\alpha,\dot{q}'_\alpha))}{\alpha<\kappa}$, where $\mathrm{pair}(\dot{q}_\alpha,\dot{q}'_\alpha)$ is the canonical name for the ordered pair $(\dot{q}_\alpha,\dot{q}'_\alpha)$. $A^*$ is a $\P \x \P$-name for a subset of $\dot{\Q}^1 \x \dot{\Q}^2$. Since $\P$ is productively $\kappa$-cc, there is a generic filter $G$ such that for some $I \sub \kappa$ of size $\kappa$, $\set{(p_\alpha,p'_\alpha)}{\alpha \in I} \sub G$. Let us write $q_\alpha^1$ and $q'_\alpha{}^2$ for $\dot{q}_\alpha^{G^1}$ and $\dot{q}_\alpha'{}^{G^2}$, $\alpha < \kappa$, respectively (where $G^1$ and $G^2$ are projections of $G$ to the coordinates). We claim that $\set{(q_\alpha^1,q'_\alpha{}^2)}{\alpha \in I}$ is an antichain, which gives a contradiction: For any $\alpha < \beta \in I$, $(p_\alpha,p'_\alpha)$ and $(p_\beta,p'_\beta)$ are compatible because they are in $G$, and their lower bound $(p,p') \in G$ forces by (\ref{eq:sq}) $(\dot{q}_\alpha,\dot{q}'_\alpha) \perp (\dot{q}_\beta,\dot{q}'_\beta)$.

For (iii), first recall the following observation: if $P_1,P_2$  forcing notions, $P_1$ is $\kappa$-cc and $P_2$ is $\kappa$-Knaster, then $P_1$ forces that $P_2$ is $\kappa$-Knaster (see Lemma 3 in \cite{C:trees} for a proof). By (ii) of the present Lemma, it suffices to show that $\P\x \P$ forces that $\dot{\Q}^1 \x \dot{\Q}^2$ is $\kappa$-cc. Let $G_1 \times G_2$ be $\P \x \P$-generic, and let us denote $(\dot{\Q}^1)^{G_1}$ by $\Q^1$ and $(\dot{\Q}^2)^{G_2}$ by $\Q^2$. $\Q^1$ is $\kappa$-Knaster in $V[G_1]$ and by the observation at the beginning of this paragraph,  $\Q^1$ is $\kappa$-Knaster in $V[G_1][G_2]$ as well. Analogously, $\Q^2$ is $\kappa$-Knaster in $V[G_2]$ and also in $V[G_2][G_1]$. It follows that both $\Q^1$ and $\Q^2$ are $\kappa$-Knaster in $V[G_1 \x G_2]$, and hence  their product $\Q^1 \x \Q^2$ is $\kappa$-Knaster.
\end{proof}

Theorem \ref{th:wKH} is slightly weaker than Fact \ref{f:tp} because it requires the productive $\kappa^+$-cc condition and not just the $\kappa^+$-cc. The reasons are technical: the tree property deals with $\kappa^{++}$-trees, and we used in \cite{HS:ind} the fact that a $\kappa^+$-cc forcing cannot add a new cofinal branch to a $\kappa^{++}$-tree; this is not true in general for trees of height $\kappa^+$ which appear in $\neg \wKH(\kappa^+)$. So we use Fact \ref{f:sq} instead.

\begin{theorem}[Indestructibility of the negation of the weak Kurepa Hypothesis] \label{th:wKH}
Assume $\omega \le \kappa < \lambda$ are cardinals, $\kappa^{<\kappa} = \kappa$ and $\lambda$ is weakly compact.  Suppose $\Add(\kappa,\lambda) * \dot{\Q}$ is productively $\kappa^+$-cc and preserves $\kappa$. Then $\M(\kappa,\lambda) * \dot{\Q}$ forces $\neg \wKH(\kappa^+)$.
\end{theorem}

\begin{proof}
Notice that we require that $\Add(\kappa,\lambda) * \dot{\Q}$ is productively $\kappa^+$-cc, and not the potentially weaker condition that $\Add(\kappa,\lambda)$ forces that $\dot{\Q}$ is productively $\kappa^+$-cc (see Lemma \ref{lm:sq}(ii)). However, in many situations this is easy to ensure: for instance if $\dot{\Q}$ is forced to be $\kappa^+$-Knaster, then $\Add(\kappa,\lambda)*\dot{\Q}$ is productively $\kappa$-cc by Lemma \ref{lm:sq} (iii)).

The proof closely follows the proof of \cite[Theorem 3.2]{HS:ind}. We will review its basic steps to make the proof self-contained. Let us denote $\M(\kappa,\lambda)$ by $\M$ and $\Add(\kappa,\lambda)$ by $R^0$.

By standard arguments, we can assume that $\Q$ has size at most $\kappa^{++}$ in $V[\M]$ and we can view it as a subset of $\kappa^{++}$ by using an isomorphic copy if necessary (a straightforward modification of Lemma 3.1 in \cite{HS:ind}  which is formulated for $\kappa^{++}$-Aronszajn trees). 

Let us first make the convention that we identify $\dot{\Q}$ with an $R^0$-name and we only consider conditions $(p,\dot{q})$ in $\M* \dot{\Q}$ in which $\dot{q}$ depends only on the $R^0$-information of $\M$.

Let $R^0 \x R^1$ denote the product $\Add(\kappa,\lambda) \x R^1_0$ from (\ref{eq:M1}). Let us fix a weakly compact embedding $j: M \to N$ with critical point $\lambda$ such that $M$ has size $\lambda$, is closed under $<\lambda$-sequences and contains all relevant parameters, in particular the forcing $(R^0 \x R^1) * \dot{\Q}$ and a name $\dot{T}$ for a weak $\kappa^+$-Kurepa tree which we assume is forced to exist for contradiction. We can further assume that $j$ itself is an element of $N$ so that a quotient analysis discussed below can be expressed in $N$ (see Fact 2.8(ii) in \cite{HS:ind}). 

Let $(\tilde{G}^0 \x \tilde{G}^1) * h^*$ be $j((R^0 \x R^1) * \dot{\Q})$-generic filter over $V$ (and hence also over $N$). Since $(R^0 \x R^1) * \dot{\Q}$ is $\lambda$-cc,  $(\tilde{G}^0 \x \tilde{G}^1) * h^*$ generates an $M$-generic filter $(G^0 \x G^1) * h$ for $(R^0 \x R^1) * \dot{\Q}$, and $\tilde{G}^0 \x \tilde{G}^1$ generates an $N$-generic filter $G^*$ for $j(\M)$ and an $M$-generic filter $G$ for $\M$ such that $j$ lifts in $N[(\tilde{G}^0 \x \tilde{G}^1)*h^*]$ to: \beq \label{eq:decompose} j: M[G][h] \to N[G^* * h^*] = N[G][h][G_Q],\eeq where $G_Q$ is a generic filter for the quotient $Q = j(\M*\dot{\Q})/G*h$. 

Our plan is to show that, \beq \label{eq:embed}\mbox{there is a projection onto } Q \mbox{ from } j(R^0 * \dot{\Q}))/(G^0 * h) \x R^1_\lambda,\eeq where $R^1_\lambda$ is the term forcing of $j(\M)/G$ from (\ref{eq:M1}). We will further show that $j(R^0 * \dot{\Q})/(G^0 * h)$ is productively $\kappa^+$-cc over $N[G][h]$ and $R^1_\lambda$ is $\kappa^+$-closed in $N[G]$ which will allow us to finish the argument.

\brm
Notice that it makes sense to consider the generic filter $G^0*h$ and the extension $N[G^0*h]$: it holds that $\dot{\Q}^G = \dot{\Q}^{G^0}$ since by our assumption $\dot{\Q}$ can be identified with an $R^0$-name.
\erm

Since $j$ is the identity on the conditions in $G$, we have \beq j''(G*h) = \set{(p,j(\dot{q}))}{p \in G \mbox{ and }\dot{q}^G \in h}.\eeq

Let us write explicitly the quotients we are going to use: \beq Q = \set{(p^*,\dot{q}^*) \in j(\M * \dot{\Q})}{N[G][h] \models \mbox{``}(p^*,\dot{q}^*) \mbox{ is compatible with }j''(G*h)\mbox{''}},\eeq where we can assume that $\dot{q}^*$ depends by elementarity only on $j(R^0)$. Further, \begin{multline} j(R^0 * \dot{\Q})/(G^0 * h) = \\ \set{(p^{*0},\dot{q}^*) \in j(R^0 * \dot{\Q})}{N[G^0][h] \models \mbox{``}(p^{*0},\dot{q}^*) \mbox{ is compatible with } j''(G^0 * h)\mbox{''}}.\end{multline} Lastly, \beq R^1_\lambda = \set{(0,p^{*1})}{N[G] \models \mbox{``}(0,p^{*1}) \mbox{ is compatible with }j''G = G\mbox{''}}.\eeq

Let us define a function $\pi: j(R^0 * \dot{\Q}))/(G^0 * h) \x R^1_\lambda \to Q$ by \beq \pi((p^{*0},\dot{q}^*),p^{*1}) = (p^*,\dot{q}^*),\eeq where $p^* = (p^{*0},p^{*1})$.

The following claim is proved as in Claim 3.4 in \cite{HS:ind}:

\begin{Claim}\label{pi}
$\pi$ is a projection from $j(R^0 * \dot{\Q})/(G^0 *h) \x R^1_\lambda$ onto $Q$.
\end{Claim}

We need the following Claim which is a variant of an analogous Claim 3.5 in \cite{HS:ind} modified in Claim \ref{claim:cc}(ii) to discuss the productively $\kappa^+$-cc condition:

\begin{Claim}\label{claim:cc}
\bce[(i)]
\item $R^1_\lambda$ is $\kappa^+$-closed in $N[G]$.
\item $j(R^0 * \dot{\Q})/G^0 *h$ is productively $\kappa^+$-cc over $N[G][h]$.
\item $\dot{\Q}^{G^0} * j(R^0 * \dot{\Q})/G^0 * \dot{h}$ is $\kappa^+$-cc over $N[G]$, where $j(R^0 * \dot{\Q})/G^0 * \dot{h}$ denotes a $\dot{\Q}^{G^0}$-name for the quotient.
\ece
\end{Claim}

\begin{proof}
For (i), the forcing $R^1_\lambda$ is $\kappa^+$-closed in $N[G]$ by (\ref{eq:M1}). 

For (ii), by elementarity, \beq j(R^0 * \dot{\Q}) \mbox{ is productively $\kappa^+$-cc over $N$.}\eeq The term forcing $R^1$ is $\kappa^+$-closed over $N$. By Easton's lemma \beq \label{eq:ccc1} j(R^0*\dot{\Q}) \mbox{ is productively $\kappa^+$-cc over $N[G^1]$.}\eeq Since $j$ restricted to $R^0 * \dot{\Q}$ is a regular embedding, $j(R^0 * \dot{\Q})$ factors over $N$ (and then also over $N[G^1]$) as \beq (R^0 * \dot{\Q}) * j(R^0 * \dot{\Q})/\dot{G}^0 * \dot{h},\eeq  where $j(R^0 * \dot{\Q})/\dot{G}^0 * \dot{h}$ is an $R^0 * \dot{\Q}$-name for the quotient. It follows by (\ref{eq:ccc1}) and Lemma \ref{lm:sq}(i) that $R^0 * \dot{\Q}$ forces in $N[G^1]$ that $ j(R^0 * \dot{\Q})/\dot{G}^0 * \dot{h}$ is productively $\kappa^+$-cc, hence $j(R^0 * \dot{\Q})/G^0 *h$ is $\kappa^+$-cc over $N[G^1][G^0*h]$.

Since there is a natural projection from $(R^0 *\dot{\Q}) \x R^1$ onto $\M * \dot{\Q}$ (analogously to the projection $\pi$ mentioned above), it follows that $j(R^0 * \dot{\Q})/G^0 *h$ is productively $\kappa^+$-cc over $N[G][h]$ as desired (since the chain condition is preserved downwards).

For (iii) first note that in this item we refer just to the $\kappa^+$-cc condition, not the productively $\kappa^+$-cc condition. Recall that $\dot{\Q}^{G^0}$ is $\kappa^+$-cc in $N[G]$ by our initial assumption. By (ii) of the present Claim, $\dot{\Q}^{G^0}$ forces over $N[G]$ that $j(R^0 * \dot{\Q})/G^0 *\dot{h}$ is $\kappa^+$-cc, hence $\dot{\Q}^{G^0} * j(R^0 * \dot{\Q})/G^0 * \dot{h}$ is $\kappa^+$-cc in $N[G]$.
\end{proof}

With Claims \ref{pi} and \ref{claim:cc}, the theorem is proved as follows. We assume for contradiction that there is in $M[G][h]$ a weak $\kappa^+$-Kurepa tree $T$ (we added into $M$ a name $\dot{T}$ for this tree at the beginning of the construction) with $\kappa^{++} = \lambda$ many cofinal branches. Then $T$ is also in $N[G][h]$. $T$ has $\kappa^{++} = j(\lambda)$ many cofinal branches in $N[G][h][G_Q]$ by the elementarity of the lifted embedding $j$ in (\ref{eq:decompose}). Since $j(\lambda)$ is inaccessible in $N[G][h]$, we also know that $T$ has $<j(\lambda)$ many cofinal branches in $N[G][h]$.

\begin{itemize}
\item Working in $N[G][h]$,  Claim \ref{claim:cc}(ii) and Fact \ref{f:sq}  imply that $j(R^0 * \dot{\Q})/G^0 *h$ cannot add new cofinal branches to $T$ over $N[G][h]$.

\item Unger essentially showed in \cite[Lemma 6]{UNGER:1} that if $P$ is $\kappa^+$-cc, $Q$ is $\kappa^+$-closed and adds a new cofinal branch to a tree $S$ in $V[P]$ whose height has cofinality at least $\kappa^+$, then $S$ must have in $V[P]$ a level of size at least $2^\kappa$.\footnote{Lemma \cite[Lemma 6]{UNGER:1}  is formulated for $\kappa^{++}$-trees $S$ and proceeds by a recursive construction of length $\kappa$ in a version of Silver's argument to show that if a new cofinal branch is added, then there must be a level of the tree $S$ of size $2^\kappa$: this gives a contradiction for our tree $T$ of size and height $\kappa^+$ as well.} Since $2^\kappa = \lambda = \kappa^{++}$  in $N[G]$, and hence also in the generic extension of $N[G]$ by the forcing $\dot{\Q}^{G^0} * j(R^0 * \dot{\Q})/G^0 * \dot{h}$, we apply Unger's lemma to the product  $$\dot{\Q}^{G^0} * j(R^0 * \dot{\Q})/(G^0 * \dot{h}) \x R^1_\lambda,$$ which has the right properties by Claim \ref{claim:cc}(i),(iii). It follows that $R^1_\lambda$ cannot add new cofinal branches to $T$ over the generic extension of $N[G][h]$ by the forcing $j(R^0 * \dot{\Q})/G^0 *h$.
\end{itemize}

It follows by these two bullets that the product \beq \label{product}  j(R^0 * \dot{\Q})/(G^0 *h ) \x R^1_\lambda \eeq does not add cofinal branches to $T$ over $N[G][h]$. However, by Claim \ref{pi}, there is a projection onto the quotient $Q$ from the product (\ref{product}), and therefore $Q$ cannot add a cofinal branch to $T$ either. It follows that $T$ cannot have $j(\lambda)$ many cofinal branches in $N[G][h][G_Q]$, which is the desired contradiction.
\end{proof}

As we discussed in Remarks \ref{rm:new1} and \ref{rm:new2}, we don't know whether $\DSS(\kappa^{++})$ holds in the Mitchell extension we use in Theorems \ref{th:P_1} and \ref{th:P_2}. However, the following indestructibility theorem will be useful for Theorem \ref{th:P_3}.

\begin{theorem}[Indestructibility of the disjoint stationary sequence property, essentially \cite{KR:DSS}] \label{th:dss}
Suppose $\kappa$ is an infinite cardinal and $\seq{s_\alpha}{\alpha \in S}$ is a disjoint stationary sequence on $\kappa^{++}$, with $S \sub \kappa^{++} \cap \cof(\kappa^+)$ stationary. Suppose $\P$ is a forcing notion which preserves $\kappa$, and moreover preserves stationary subsets of both $\kappa^+$ and $\kappa^{++}$. Then $\P$ forces that $\seq{s_\alpha}{\alpha \in S}$ is a disjoint stationary sequence on $\kappa^{++}$.
\end{theorem}

\begin{proof}
Since $\P$ preserves  stationary subsets of $\kappa^{++}$, $S$ is still stationary. It suffices to check that if $\P$ preserves stationary subsets of $\kappa^+$, it preserves stationary subsets of $\mathcal{P}_{\kappa^+}(\alpha)$ for $\alpha \in S$. Let $\seq{x_i}{i<\kappa^+}$ be an increasing and continuous sequence of subsets of $\alpha$ of size $<\kappa^+$  whose union is $\alpha$; then it is easy to verify that $s$ is stationary in $\mathcal{P}_{\kappa^+}(\alpha)$ iff $\set{i < \kappa^+}{x_i \in s}$ is stationary in $\kappa^+$. It follows that if $\P$ preserves stationary subsets of $\kappa^+$, it also preserves stationary subsets of $\mathcal{P}_{\kappa^+}(\alpha)$.
\end{proof}

Note that \cite[Corollary 2.2]{GK:a} gives an indestructibility result for $\neg \AP(\kappa^{++})$, $\kappa$ regular, but it holds only for $\kappa$-centered forcings,\footnote{A forcing is $\kappa$-centered if it can be written at the union of $\kappa$-many filters, in particular it is $\kappa^+$-cc.} and thus is not sufficient for our purposes.

\section{A review of the argument which makes $\uu$ small}\label{sec:review}

\subsection{The original forcing}\label{sec:orig}

Let us briefly review the definition of the forcing $\P$ introduced in \cite{F:u} by Brooke-Taylor, Fischer, Friedman and Montoya. We first sumarize its basic structure and then give a precise statement in Definition \ref{def:P}. 

Let $\kappa$ be a Laver-indestructible supercompact cardinal, and $\mu \ge \kappa^{++}$ a cardinal of cofinality $>\kappa$ with $\mu^\kappa = \mu$. We view that the ordinals below $\mu^+$ as divided into three disjoint cofinal subsets which are reserved for three different tasks: the first subset $I_0$ is reserved for the Mathias forcing which controls the ultrafilter number $\mf{u}(\kappa)$, the second subset $I_1$ is used to control the minimal size of a $\kappa$-mad family $\mf{a}(\kappa)$, and the third subset $I_2$ is used to control the pseudo-intersection number $\mf{p}(\kappa)$. For concreteness, $I_0$ will be the set of even ordinals (including limit ordinals) below $\mu^+$ and $I_1$ and $I_2$  some sets of odd ordinals such that both $I_1$ and $I_2$ are cofinal in every limit ordinal (and $I_1 \cup I_2$ contain all odd ordinals in $\mu^+$). The iteration on the segments $I_1$ and $I_2$ is defined in the usual way with $<\kappa$ support, but $I_0$ is more complicated because it uses a lottery among forcings and two types of support to control $\mf{u}(\kappa)$.

The precise definition of the forcing  follows the definition at the beginning of Section 3 of \cite{F:u}. As defined, $\P$  allows any $\kappa$-centered, $\kappa$-directed closed forcing $\dot{\Q}_\alpha$ of size at most $\mu$ at odd stages. By using forcings to control minimal size of mad-families and the pseudo-intersection number we obtain the specific forcing we need (details are in the proof of \cite[Theorem 35]{F:u}).

\begin{definition}\label{def:P}
Let $\kappa$ be a Laver-indestructible supercompact cardinal, and $\mu \ge \kappa^{++}$ a cardinal of cofinality $>\kappa$ with $\mu^\kappa = \mu$. Then $\P = \seq{(\P_\alpha,\dot{\Q}_\beta)}{\alpha \le \mu^+, \beta < \mu^+}$ is defined by recursion on $\mu^+$ as follows:
\end{definition}

\begin{itemize}
\item (Lottery, $I_0$.) If $\alpha$ is an even ordinal ($\alpha \in \EVEN$), let $\mx{NUF}_\alpha$ be the set of normal ultrafilters on $\kappa$ in $V[\P_\alpha]$. Then let $\dot{\Q}_\alpha$ be a name for the poset with underlying set of conditions $$\{1_{\Q_\alpha}\} \cup \set{\{U\} \x \M^\kappa_U}{U \in \mx{NUF}_\alpha},$$ where $1_{\Q_\alpha}$ is denotes the greatest element of $\dot{\Q}_\alpha$, and extension relation is defined by $q\le p$ if and only if $p = 1_{\Q_\alpha}$ or there is $U \in \mx{NUF}_\alpha$ such that
$$p =(U,p_1), q = (U,q_1) \mbox{ and } q_1 \le_{\M^\kappa_U} p_1,$$ where $\M^\kappa_U$ is a generalized Mathias forcing with a $\kappa$-complete ultrafilter $U$, as defined in \cite[Definition 3]{F:u}.

\item ($I_1,I_2$.) If $\alpha$ is an odd ordinal ($\alpha \in \ODD$) then let $\dot{\Q}_\alpha$ be a name for a $\kappa$-centered, $\kappa$-directed closed forcing of size at most $\mu$.

\item (Supports for lottery on $I_0$.)

\begin{itemize}
\item \emph{Ultrafilter support of $p$}, denoted $\USupt(p)$, is defined as $$\USupt(p) = \set{\alpha \in \dom{p} \cap \EVEN}{\restr{p}{\alpha} \Vdash p(\alpha) \neq 1_{\Q_\alpha}}.$$ The support $\USupt(p)$ is required to be bounded below $\mu^+$ and closed downwards in the sense that if $\alpha \in \USupt(p)$ and $\beta \in \EVEN \cap \alpha$, then $\beta \in \USupt(p)$. 

\item \emph{Essential support of $p$}, denoted $\SSupt(p)$, is defined as the set of all $\alpha \in \dom{p} \cap \EVEN$ such that $$\neg (\restr{p}{\alpha} \Vdash p(\alpha) \in \{1_{\Q_\alpha}\} \cup \set{(U,1_U)}{U \in \mx{NUF}_\alpha}),$$ where $1_U$ is the greatest condition in $\M^\kappa_U$. The support $\SSupt(p)$ is required to have size $<\kappa$ and to be included as a subset in the supremum of $\USupt(p)$, and hence $\SSupt(p) \sub \USupt(p)$ by the downwards closure of $\USupt(p)$. Note that by definition, if $\alpha \in \USupt(p) \setminus \SSupt(p)$, then $\restr{p}{\alpha} \Vdash p(\alpha) = (U,1_U)$ for some $U \in \mx{NUF}_\alpha$.
\end{itemize}

\item (Supports for $I_1,I_2$.) \emph{Direct support of $p$}, denoted $\RSupt(p)$, is defined as the set of all $\alpha \in \dom{p} \cap \ODD$ such that
$$
\neg (\restr{p}{\alpha} \Vdash p(\alpha) = 1_{\Q_\alpha}).
$$
The support $\RSupt(p)$ is required to have size $<\kappa$ and to be a subset of the supremum of $\USupt(p)$.
\end{itemize}

If $\SSupt(p) = \emptyset$ and $\USupt(p)  = \gamma \cap \EVEN$, we write $\lot{\gamma}$ to indicate that $p$ has made its choice regarding the normal ultrafilters below $\gamma$, but has not chosen any (non-trivial) conditions in the respective Mathias forcings.

We will work with restrictions of the form $\P_\delta\da \lot{\delta}$ (the conditions in $\P_\delta$ which extend $\lot{\delta}$) which always have a dense subset of size at most $\mu$ and are $\kappa^+$-Knaster (unlike $\P_\delta$ which has large antichains due to the lottery over all normal ultrafilters on $\kappa$).\footnote{The article \cite{F:u} only states that $\P_\delta\da \lot{\delta}$ is $\kappa^+$-cc, but it is easy to see that it is $\kappa^+$-Knaster: This forcing has $<\kappa$-supports $\SSupt$ and $\RSupt$, and all iterands are $\kappa$-directed closed and $\kappa$-centered: their compatibility is determined by stems $s \in \kappa^{<\kappa}$ (the Mathias forcing on $I_0 \cup I_2$ and an almost-disjointness forcing on $I_1$). To argue that $\P_\delta\da \lot{\delta}$ is $\kappa^+$-Knaster, first use a $\Delta$-system lemma on the supports and then use the $\kappa$-centeredness of the iterands on the root of the system (the stems $s$ can be without the loss of generality represented by checked names).} 

The key idea in \cite{F:u} is to identify a suitable name $\UU$ for a normal ultrafilter on $\kappa$ in $V[\P]$, and truncate $\P$ at some ordinal $\delta \in (\mu,\mu^+)$ of the required cofinality such that for a certain condition $p_{\dot{U}}$ (which is of the form $\lot{\delta}$), the forcing $\P_\delta \da p_{\dot{U}}$ forces $\kappa^+ < \uu < 2^\kappa = \mu$. The ultrafilter number $\uu$ will be equal to the cofinality of $\delta$.

The main tool in the proof is to argue that any normal ultrafilter on $\kappa$ in $V[\P]$ reflects sufficiently often below $\mu^+$. Lemma \ref{lm:fix1} below is implicit in \cite[Lemma 10]{F:u}.

We say that $S \sub \mu^+$ is a $\cof(\bigger\kappa)$-club if $S$ is unbounded in $\mu^+$ and closed at limit points of cofinality $\bigger \kappa$.

\begin{lemma}\label{lm:fix1}
Assume $\kappa$, $\mu$ and $\P$ are is in Definition \ref{def:P}. Assume that \beq \label{1}1_\P \Vdash \UU \mbox{ is a normal ultrafilter on $\kappa$}.\eeq Then there is a $\mx{cof}(\bigger\kappa)$-club $S_{\UU} \sub \mu^+$ concentrating on limit ordinals (hence on $\EVEN$) where $\UU$ reflects. More precisely, there is a decreasing sequence $\seq{\lot{\alpha}}{\alpha \in S_{\UU}}$ continuous at points of cofinality $\bigger\kappa$ which chooses in the lottery the restrictions of the ultrafilter $\UU$ at the relevant stages:
\begin{enumerate}[(i)]
\item For every $\alpha \in S_{\UU}$, \beq \label{fix1} \lot{\alpha} \Vdash\UU \cap V[\P_\alpha] \in V[\P_\alpha].\eeq 
\item For all $\alpha < \alpha^* \in S_{\UU}$ \beq \label{fix2*} \lot{\alpha^*} \Vdash \dot{U}_\alpha = \UU \cap V[\P_\alpha] \in V[\P_\alpha],\eeq where  $\dot{U}_\alpha$ is a name of the ultrafilter selected by the lottery at stage $\alpha$ by $\lot{\alpha^*}$, i.e.\ $\lot{\alpha^*} \Vdash \la \dot{U}_\alpha, 1_{\dot{U}_\alpha} \ra \in \Q_\alpha$.
\item The sequence is continuous at ordinals of cofinality $\bigger \kappa$: for any limit $\delta$ in $S_{\UU}$ of cofinality $\bigger\kappa$, $\lot{\delta}$ is the infimum of $\seq{\lot{\beta}}{\beta < \delta}$ such that (\ref{fix2*}) holds for all $\alpha<\alpha^* < \delta$.
\end{enumerate}
\end{lemma}

\begin{proof}
The construction of $S_{\UU} = \set{\alpha_\xi}{\xi < \mu^+}$ is by transfinite recursion, repeating at successor stages the proof of \cite[Lemma 10]{F:u}: if $\alpha_\xi$ is constructed, Lemma 10 yields $\alpha_\xi < \alpha_{\xi+1}$ and $\lot{\alpha_{\xi+1}}$ such that $\lot{\alpha_{\xi+1}}$ forces $\dot{U} \cap V[\P_{\alpha_{\xi+1}}] \in V[\P_{\alpha_{\xi+1}}]$. The argument for this stage is itself a recursive construction which diagonalizes over maximal antichains and nice names $\dot{x}$ for subsets of $\kappa$ deciding whether $\dot{x}$ belongs to $\dot{U}$ or not (the set of conditions in $\P_{\alpha_{\xi+1}} \da \lot{\alpha_{\xi+1}}$ which decide $\dot{x} \in \dot{U}$ is predense below $\lot{\alpha_{\xi+1}}$). To address (\ref{fix2*}), define an extension $\lot{\alpha_{\xi+1}+1} = \lot{\alpha_{\xi+1}} \conc p(\alpha_{\xi+1})$, where $p(\alpha_{\xi+1})$ is of the form $\la \dot{U}_{\alpha_{\xi+1}}, 1_{\dot{U}_{\alpha_{\xi+1}}} \ra$, where $\dot{U}_{\alpha_{\xi+1}}$ is a name for $\dot{U} \cap V[\P_{\alpha_{\xi+1}}]$. 

If $\seq{\alpha_\xi}{\xi < \zeta}$ is constructed and $\zeta$ has cofinality $>\kappa$, set $\alpha_\zeta = \sup\set{\alpha_\xi}{\xi < \zeta}$ and let $\lot{\alpha_\zeta}$ be the infimum of $\seq{\lot{\alpha_\xi}}{\xi<\zeta}$. If $\dot{x}$ is a nice name for a subset of $\kappa$ in $V[\P_{\alpha_\zeta}\da \lot{\alpha_\zeta}]$, it belongs to an initial segment $V[\P_{\alpha_\eta}\da \lot{\alpha_\eta}]$ for some $\eta < \zeta$ since $\zeta$ has cofinality $>\kappa$. It follows that the set of conditions in $\P_{\alpha_{\eta+1}} \da \lot{\alpha_{\eta+1}}$ which decide whether $\dot{x}$ is in $\dot{U}$ is predense below $\lot{\alpha_{\eta+1}}$, and hence also below $\lot{\alpha_\zeta}$. Thus $\lot{\alpha_\zeta}$ forces  $\dot{U} \cap V[\P_{\alpha_\zeta}] \in V[\P_{\alpha_\zeta}]$. It follows that $S_{\UU}$ is continuous at limit stages of cofinality $>\kappa$. 
\end{proof}

\subsection{Obtaining a suitable  normal ultrafilter for the proof}\label{sec:right}

The existence of normal ultrafilters in $V[\P]$ follows by assuming that $\kappa$ is Laver-indestructibly supercompact. In order to find a suitable name $\dot{U}$ for which the iteration on $I_0$ -- below some well-chosen condition $p_{\UU}$ -- generates a base of a uniform ultrafilter and ensures the desired value of $\mf{u}(\kappa)$, one needs to go into the details of the Laver preparation. We follow the structure of the argument in \cite{F:u} while we additionally include the forcings $\M(\kappa,\lambda)$ and $\Add(\kappa,\lambda)$ in the preparation (we also provide a clarification of certain points in the proof; see Footnote \ref{ft:un1}). 

In preparation for Theorems \ref{th:P_1} and \ref{th:P_2}, we need to consider not just the forcing $\P$ from Definition \ref{def:P}, but forcings of the form \beq \label{eq:modM} \M(\kappa,\lambda) * \dot{\P}\eeq and \beq \label{eq:modA} \Add(\kappa,\lambda)*\dot{\P},\eeq where $\dot{\P}$ is a name for a forcing from Definition \ref{def:P}. The argument is the same in both cases, so let us write this forcing as \beq \label{eq:mod} P * \dot{\P},\eeq with the understanding that $P$ is either $\Add(\kappa,\lambda)$ or $\M(\kappa,\lambda)$.

Suppose $\kappa$ is supercompact in some universe $V'$, and $L$ is the Laver preparation for $\kappa$ which makes $\kappa$ in $V'[L]$ indestructible under all $\kappa$-directed closed forcings. Let $H*F*G$ be an $L*\dot{P}*\dot{\P}$-generic over $V'$, where $\dot{P}$ is an $L$-name for $\Add(\kappa,\lambda)$ or $\M(\kappa,\lambda)$, and $\dot{\P}$ is an $L*\dot{P}$-name for $\P$ in Definition \ref{def:P}.

Fix a suitable supercompact embedding \beq j: V' \to M \eeq with critical point $\kappa$ such that in $M$ the iteration $j(L)$ on $\kappa+1$ is equal to $L*\dot{P}*\dot{\P}$.

Using a standard master condition argument we can lift in $V'[H][F][G]$ further to $j'$ \beq \label{c0} j':V'[H][F] \to M[H][F][G][H^*][F^*],\eeq where $H^*$ is any generic over $M[H][F][G]$ for the tail of $j(L)$ and $F^*$ is any generic over $M[H][F][G][H^*]$ containing a master condition for $F$.

For the lifting of $j'$ to $j^*$, we need a specific master condition $p^*$ (see Remark \ref{rm:p*} for details) extending the pointwise image of $j'{}"G$. Since $p^*$ is a master condition, $j'$ lifts to $j^*$: \beq \label{c1} j^*: V'[H][F][G] \to M[H][F][G][H^*][F^*][G^*], \eeq where $G^*$ contains $p^*$. 

In $V'[H][F]$, let $\UU$ be a $\P$-name forced by $1_{\P}$ to be a normal ultrafilter on $\kappa$ generated by $j^*$. This name $\UU$ is the one to which Lemma \ref{lm:fix1} is applied, and which determines the ordinal $\delta \in S_{\dot{U}}$, the condition $p_{\dot{U}} = \lot{\delta}$, and finally the forcing $\P_\delta \da p_{\dot{U}}$.\footnote{\label{ft:un1} The argument in \cite[Theorem 12]{F:u} seems to suggest that one can start with an arbitrary name $\dot{U}$ and apply \cite[Lemma 10]{F:u} with it, and control the interpretation of $\dot{U}$ by choosing the right master condition to lift $j'$ to $j^*$ in (\ref{c1}). But $\UU$ is a $\P$-name and the lifting from $j'$ to $j^*$ does not effect its interpretation which is fixed by the $\P$-generic filter $G$ (and hence it is not clear why $\UU^G$ should contain the Mathias generic subsets of $\kappa$ which ensure a base of size $\kappa^*$). Instead, we argue that $1_{\P}$ forces over $V'[H][F]$ that there is a normal ultrafilter $\UU$ with the required properties (*) reviewed in Remark \ref{rm:p*}, apply Lemma \ref{lm:fix1} with this name $\dot{U}$ to secure a condition $p_{\dot{U}}$ of the form $\lot{\delta}$, and only then choose a generic filter $G$ for $\P_\delta$ which contains the condition $p_{\dot{U}}$.}

\brm \label{rm:p*}
Let us briefly review the definition of $p^*$. We work in the generic extension $V'[H][F][G]$, but all we say can be translated into $\P$-forcing statements dealing with names for $j'$ and $G$ over $V'[H][F]$. Let $p_0^*$ be some master condition for $j'[G]$; we extend $p_0^*$ by some $p^* \le p_0^*$ which in addition has the following property (*):

\medskip
\begin{itemize}
\item[(*)] Whenever $\alpha < \mu^+$ has the property that for every $A \in (\UU_\alpha)^G$ (where $\UU_\alpha$ is a name for the normal ultrafilter selected by the lottery at stage $\alpha$ by a condition in $G$) there is some name $\dot{A}$ for $A$ and a condition $p_A \in G_\alpha$ with $j'(p_A) \Vdash \kappa \in j'(\dot{A})$, then $p^*(j(\alpha))$ is obtained from $p^*_0(j(\alpha))$ by adding $\{\kappa\}$ into the stem of the Mathias forcing at stage $j(\alpha)$.
\end{itemize}

\medskip
Let $j^*$ be the lifting of $j'$ with any $G^*$ which contains $p^*$, and let $\dot{U}$ be a name for the normal ultrafilter generated by $j^*$. Let us apply Lemma \ref{lm:fix1} with this name $\dot{U}$, obtaining  $S_{\dot{U}}$ and the condition $\lot{\delta}= p_{\dot{U}}$ for an appropriate $\delta$. Let $G$ be any $\P_{\delta} \da p_{\dot{U}}$-generic filter. The Mathias-generic subsets of $\kappa$, denoted $x_\alpha$, at all stages $\alpha \in S_{\UU}$ are in $\UU^G$ because $p_{\dot{U}} \in G$ and $p^*(j(\alpha))$ was defined to contain the stem $x_\alpha \cup \{\kappa\}$ (and $\kappa \in j^*(x_\alpha)$ is equivalent to $x_\alpha$ being in $\UU^G$).
\erm

\section{Compactness and generalized cardinal invariants}\label{sec:main}

Using the indestructibility result for $\SR(\kappa^{++})$ reviewed in Section \ref{sec:prelim} and the properties of the forcing $\P_\delta \da p_{\UU}$ reviewed earlier, one can immediately observe the following (see Definitions \ref{def:u} and \ref{def:t} for $\uu$ and $\mf{t}(\kappa)$ and \cite{F:u} for the remaining cardinal invariants):

\begin{theorem}\label{th:P_1}
Suppose $\kappa$ is a supercompact cardinal, $\lambda > \kappa$ is a weakly compact cardinal, $\mu \ge \lambda$ is a cardinal with cofinality $>\kappa$ with $\mu^\kappa = \mu$, and $\kappa^*$ is a regular cardinal with $\kappa < \kappa^* < \mu$. Then there is a generic extension $V[G]$ which satisfies the following:
\bce[(i)]
\item Exactly the cardinals in the open interval $(\kappa,\lambda)$ are collapsed, with $\lambda = (\kappa^{++})^{V[G]}$,
\item $2^\kappa = \mu$,
\item $\SR(\kappa^{++})$.
\ece
\medskip
And the following identities hold:
\begin{multline}\label{eq:ci1}
\kappa^* = \mf{p}(\kappa) = \mf{t}(\kappa) = \mf{b}(\kappa) = \mf{d}(\kappa) = \mf{s}(\kappa) = \mf{r}(\kappa) = \mf{a}(\kappa) =  \mf{u}(\kappa) = \\ = \mx{cov}(\mathcal M_\kappa)
= \mx{add}(\mathcal M_\kappa) = \mx{non}(\mathcal M_\kappa) = \mx{cof}(\mathcal M_\kappa).
\end{multline}
\end{theorem}

\begin{proof}

Let us define \beq \label{eq:P1}\P_1 := \M(\kappa,\lambda) * \dot{\P}_\delta \da p_{\UU},\eeq where $\dot{\P}_\delta \da p_{\UU}$ is defined in $V[\M(\kappa,\lambda)]$ following the review in Section \ref{sec:orig} with the condition $p_{\UU}$ determined by Lemma \ref{lm:fix1} with respect to a name $\dot{U}$ for a normal ultrafilter on $\kappa$ in $V[\M(\kappa,\lambda)*\dot{\P}]$ obtained through the Laver preparation anticipating the forcing (\ref{eq:modM}), and the construction described in Remark \ref{rm:p*}. In particular $\delta \in (\mu,\mu^+)$ is such that $S_{\dot{U}} \cap \delta$ has cofinality $\kappa^*$ and $p_{\UU}$ is equal to $\lot{\delta}$ in the construction in Lemma \ref{lm:fix1} applied with $\dot{U}$.

Let $G = F * G_\delta$ be an $\M(\kappa,\lambda)* \dot{\P}_\delta \da p_{\UU}$-generic filter.

The principle $\SR(\kappa^{++})$ holds in $V[F]$ by Fact \ref{f:M} and by Fact \ref{f:sr} continues to hold in $V[G]$ because $(\dot{\P}_\delta \da p_{\UU})^F$ is $\kappa^+$-cc in $V[F]$.

The desired pattern of the cardinal invariants follows exactly as in \cite{F:u}.
\end{proof}

We may obtain $\SR(\kappa^{++})$ in a different way, and in addition have also $\TP(\kappa^{++})$ and $\neg \wKH(\kappa^+)$, if we modify forcing $\P_1$ from (\ref{eq:P1}). This modification results in a different pattern of cardinal invariants:

\begin{theorem}\label{th:P_2}
Suppose $\kappa$ is a supercompact cardinal, $\lambda > \kappa$ is a weakly compact cardinal, $\mu \ge \lambda$ is a cardinal with cofinality $>\kappa$ with $\mu^\kappa = \mu$, and $\kappa^*$ is a regular cardinal with $\kappa < \kappa^* < \mu$. Then there is a generic extension $V[G]$ which satisfies the following:
\bce[(i)]
\item Exactly the cardinals in the open interval $(\kappa,\lambda)$ are collapsed, with $\lambda = (\kappa^{++})^{V[G]}$,
\item $2^\kappa = \mu$,
\item $\SR(\kappa^{++})$, $\TP(\kappa^{++})$ and $\neg \wKH(\kappa^+)$.
\ece
\medskip
And the following hold:
\begin{multline}\label{eq:ci12}
\kappa^+ = \mf{p}(\kappa) = \mf{t}(\kappa) \le \\ \le \kappa^* = \mf{b}(\kappa) = \mf{d}(\kappa) = \mf{s}(\kappa) = \mf{r}(\kappa) = \mf{a}(\kappa) =  \mf{u}(\kappa) = \mf{r}(\kappa) = \\ = \mx{cov}(\mathcal M_\kappa)
= \mx{add}(\mathcal M_\kappa) = \mx{non}(\mathcal M_\kappa) = \mx{cof}(\mathcal M_\kappa).
\end{multline}
\end{theorem}

\begin{proof}
Let us define \beq \label{eq:P2} \P_2 := \Add(\kappa,\lambda) * (\dot{\P}_\delta \da p_{\UU} \x \dot{R}),\eeq 
where $\dot{\P}_\delta \da p_{\UU}$ is defined in $V[\Add(\kappa,\lambda)]$ following the review in Section \ref{sec:orig} with the condition $p_{\UU}$ determined by Lemma \ref{lm:fix1} with respect to a name $\dot{U}$ for a normal ultrafilter on $\kappa$ in $V[\Add(\kappa,\lambda)*\dot{\P}]$ obtained through the Laver preparation anticipating the forcing (\ref{eq:modA}), and the construction described in Remark \ref{rm:p*}. In particular $\delta \in (\mu,\mu^+)$ is such that $S_{\dot{U}} \cap \delta$ has cofinality $\kappa^*$ and $p_{\UU}$ is equal to $\lot{\delta}$ in the construction in Lemma \ref{lm:fix1} applied with the name $\dot{U}$. The forcing $\dot{R}$ is forced by $\Add(\kappa,\lambda)$ to be $\kappa^+$-distributive, with $\Add(\kappa,\lambda)*\dot{R}$ being forcing equivalent to $\M(\kappa,\lambda)$ (see Definition \ref{def:M} for more details).

Let $F = F_0 * F_1$ be $\Add(\kappa,\lambda)*\dot{R}$-generic and let $G_\delta$ be $\P_\delta \da p_{\UU} := (\dot{\P}_\delta \da p_{\UU})^{F_0}$-generic over $V[F_0]$. By Lemma \ref{lm:mutual}(ii), $F_0 * (G_\delta \x F_1)$ is $\P_2$-generic and satisfies the following:

\begin{lemma}\label{lm:mutual} The following hold:
\begin{enumerate}[(i)] 
\item $\M(\kappa,\lambda)$ forces $\dot{\P}_\delta \da p_{\UU}$ is productively $\kappa^+$-cc.
\item Suppose $F_1$ is $R = \dot{R}^{F_0}$-generic over $V[F_0]$ and $G_\delta$ is $\P_\delta \da p_{\UU} = (\dot{\P}_\delta \da p_{\UU})^{F_0}$-generic over $V[F_0]$. Then $F_1$ and $G_\delta$ are mutually generic over $V[F_0]$.
\item $F_1$ does not add new $\kappa$-sequences over $V[F_0][G_\delta]$.
\end{enumerate}
\end{lemma}

\begin{proof}
For (i), it suffices to show that $\Add(\kappa,\lambda) \x R^1$ forces $\dot{\P}_\delta \da p_{\UU}$ is productively $\kappa^+$-cc, because there is a projection from $\Add(\kappa,\lambda) \x R^1$ onto $\M(\kappa,\lambda)$ (see (\ref{eq:M1}) for more details). $\Add(\kappa,\lambda) * \dot{\P}_\delta \da p_{\UU}$ is productively $\kappa^+$-cc because it is $\kappa^+$-Knaster, and by Easton's lemma the $\kappa^+$-closed $R^1$ forces $\Add(\kappa,\lambda)*\dot{\P}_\delta \da p_{\UU}$ is still productively $\kappa^+$-cc, and hence $R^1 \x \Add(\kappa,\lambda)$ forces $\dot{\P}_\delta \da p_{\UU}$ is productively $\kappa^+$-cc (compare with Lemma \ref{lm:sq}(i)).

For (ii), suppose $F_1$ is $V[F_0]$-generic for $R$, and $G_\delta$ is $\P_\delta \da p_{\UU}$-generic over $V[F_0]$. It suffices to show that $G_\delta$ is generic over the model $V[F_0][F_1]$ because then $F_1 \x G_\delta$ is a generic filter over $V[F_0]$ for the product $\P_\delta \da p_{\UU} \x R$, and hence $V[F_0][F_1][G_\delta] = V[F_0][G_\delta][F_1]$. By (i), $\P_\delta \da p_{\UU}$ is still $\kappa^+$-cc over $V[F_0][F_1]$, and since $R$ does not add new $\kappa$-sequences of elements over the model $V[F_0]$, it follows that $\P_\delta \da p_{\UU}$ has the same maximal antichains in $V[F_0]$ as it has in $V[F_0][F_1]$, and so $G_\delta$ is generic over the larger model $V[F_0][F_1]$ as well.

For (iii), suppose for contradiction that $x$ is a $\kappa$-sequence in $V[F_0][G_\delta][F_1]$ which is not in $V[F_0][G_\delta]$. This means that there is a $\P_\delta \da p_{\UU}$-name $\dot{x}$ in $V[F_0][F_1]$ such that $\dot{x}^{G_\delta} = x$ in $V[F_0][G_\delta][F_1] = V[F_0][F_1][G_\delta]$ and $\dot{x}$ is not in $V[F_0]$. However, this name is itself a $\kappa$-sequence of elements in $V[F_0]$ because $\P_\delta \da p_{\UU}$ is $\kappa^+$-cc. This is a contradiction because $F_1$ does not add new $\kappa$-sequences over $V[F_0]$.
\end{proof}

\brm
The argument in Lemma \ref{lm:mutual}(iii) actually shows the following more general claim: Suppose $\P$ and $\Q$ are forcing notions with $\P$ being $\kappa^+$-cc and $\Q$ being $\kappa^+$-distributive. Then if $\Q$ forces that $\P$ is $\kappa^+$-cc, then $\P$ forces that $\Q$ is $\kappa^+$-distributive.
\erm

Let us denote the filter $F_0 * (G_\delta \x F_1)$ by $G$. Fact \ref{f:M}, Lemma \ref{lm:mutual}(i),  Facts \ref{f:sr} and \ref{f:tp} and Theorem \ref{th:wKH}  imply that $\SR(\kappa^{++})$, $\TP(\kappa^{++})$ and $\neg \wKH(\kappa^+)$, respectively, hold in $V[G]$.

It remains to check that the required pattern of cardinal invariants in (\ref{eq:ci12}) holds in $V[G]$. Exactly as in Theorem \ref{th:P_1}, the identities (\ref{eq:ci1}) hold in $V[F_0][G_\delta]$. By Lemma \ref{lm:mutual}(iii), the quotient forcing $R$ does not add new sequences of length $\kappa$, so the space ${}^\kappa \kappa$ is the same in $V[F_0][G_\delta]$ and $V[G]$. However, $R$ may add new subsets of $^{\kappa}\kappa$ and consequently change the values of some cardinal invariants. We will argue that this happens only for $\mf{p}(\kappa)$ and $\mf{t}(\kappa)$ (from the list of the invariants we mention) which will be equal to $\kappa^+$ in $V[G]$ disregarding the value of $\kappa^*$ (see Lemma \ref{lm:tower}). 

Let us start by showing that all cardinal invariants in our list except for $\mf{p}(\kappa)$ and $\mf{t}(\kappa)$ (if $\kappa^*>\kappa^+$) continue to have value $\kappa^*$ in $V[G]$:

\begin{lemma}
The following identities hold in $V[G]$:
\begin{multline}\label{eq:vg}
\kappa^* = \mf{b}(\kappa) = \mf{d}(\kappa) = \mf{s}(\kappa) = \mf{r}(\kappa) = \mf{a}(\kappa) =  \mf{u}(\kappa) = \\ \mx{cov}(\mathcal M_\kappa)
= \mx{add}(\mathcal M_\kappa) = \mx{non}(\mathcal M_\kappa) = \mx{cof}(\mathcal M_\kappa).
\end{multline}
\end{lemma}
\begin{proof}

Let us first focus on the invariants $\mf{b}(\kappa), \mf{d}(\kappa), \mf{s}(\kappa), \mf{r}(\kappa), \mf{a}(\kappa), \mf{u}(\kappa)$. We know that they are all equal to $\kappa^*$ in  $V[F_0][G_\delta]$, and this fact is witnessed for each invariant by an appropriate subset of $({}^\kappa \kappa)^{V[F_0][G_\delta]}$ of size $\kappa^*$. Since $({}^\kappa \kappa)^{V[F_0][G_\delta]} = ({}^\kappa \kappa)^{V[G]}$, these witnesses are still relevant and imply $$\mf{b}(\kappa), \mf{d}(\kappa), \mf{s}(\kappa), \mf{r}(\kappa), \mf{a}(\kappa), \mf{u}(\kappa) \le \kappa^* \mbox{ in $V[G]$}.$$
It thus suffices to show $\kappa^* \le \mf{b}(\kappa), \mf{d}(\kappa), \mf{s}(\kappa), \mf{r}(\kappa), \mf{a}(\kappa), \mf{u}(\kappa)$.

As the Mathias generic subsets of $\kappa$ are added cofinally often below $\delta$, we have $$\kappa^* \le \mf{b}(\kappa) \mbox{ and } \kappa^* \le \mf{s}(\kappa)$$ in $V[G]$: If an unbounded or a splitting family $B$ of size $<\kappa^*$ were added by $R$, then because $\delta$ has cofinality $\kappa^*$ in $V[G]$, it would follow \begin{equation} \label{eq:b} V[G] \models B \sub ({}^\kappa \kappa)^{V[F_0][G_\alpha]} \end{equation} for some $\alpha < \delta$ -- but this is impossible because a Mathias generic subset of $\kappa$ added at any stage after $\alpha$ dominates all functions in $V[F_0][G_\alpha]$ and is unsplit by any subset of $\kappa$ in $V[F_0][G_\alpha]$. 

Since $\mf{b}(\kappa) = \kappa^*$, the remaining inequalities $\kappa^* \le \mf{d}(\kappa), \mf{r}(\kappa), \mf{u}(\kappa), \mf{a}(\kappa)$ follow by $\ZFC$ inequalities (\ref{eq1})--(\ref{eq3}) which were proved in \cite{F:u}:
\begin{equation} \label{eq1}
\kappa^+ \le \mf{b}(\kappa) \le \mf{a}(\kappa),
\end{equation}
\begin{equation} \label{eq2}
\mf{b}(\kappa) \le \mf{r}(\kappa) \le \mf{u}(\kappa),
\end{equation}
\begin{equation}\label{eq3}
\mf{b}(\kappa) \le \mf{d}(\kappa).
\end{equation}

Let us now turn to the invariants related to the meager ideal $\mathcal M_\kappa$. 
We first observe that $\P_\delta \da p_{\UU}$ adds a sequence of Cohen generic subsets of $\kappa$ of order-type $\kappa^*$, $\seq{c_\beta}{\beta< \kappa^*}$, which are cofinal in $\delta$. By standard arguments, for every nowhere dense set $A$ in $\kappa^\kappa$ in $V[F_0][G_\delta]$, there is some $\alpha <\kappa^*$ such that for every $\beta>\alpha$, $c_\beta \not \in A$. This implies $\mx{cov}(\mathcal M_\kappa) \ge \kappa^*$ and also $\mx{non}(\mathcal M_\kappa) \le \kappa^*$ (because $\set{c_\beta}{\beta<\kappa^*}$ is seen to be non-meager) in $V[F_0][G_\delta]$. For the proof, see for instance \cite[Proposition 47]{Bat:gen} which works in our case with obvious modifications. We need to argue that both inequalities still hold in $V[G]$. By \cite[Proposition 47]{Bat:gen}, one can work with (closed) nowhere dense sets $A_f$ determined by certain functions $f: 2^{<\kappa} \to 2^{<\kappa}$ (because for every nowhere dense set $D$ there is $f$ such that $D \sub A_f$).\footnote{To avoid a possible confusion: In \cite[Proposition 47]{Bat:gen}, $\set{c_\beta}{\beta< \kappa^+}$ is shown to be non-meager. This is because in that article only the Cohen forcing at $\kappa$ is used, and any $\kappa$-many $A_f$'s have a name which uses only $<\kappa^+$-many Cohen coordinates, which leaves some $c_\beta$, $\beta < \kappa^+$, outside of these $A_f$'s. For our iteration, this product-type argument is not possible (all we can say is that every $A_f$ appears in the iteration by some stage $\alpha < \delta$), so only $\set{c_\beta}{\beta< \kappa^*}$ is seen to be non-meager.}  Since $R$ does not add new $\kappa$-sequences, these functions $f$ and the associated nowhere dense sets $A_f$ are the same in $V[F_0][G_\delta]$ and $V[G]$. Now we use a similar argument which we used to show that $\kappa^*$ is less or equal to $\mf{b}(\kappa)$ in (\ref{eq:b}): If $B$ is a collection of $<\kappa^*$-many nowhere dense sets $A_f$ in $V[G]$, then $B$ is contained as a subset in $V[F_0][G_\alpha]$ for some $\alpha < \delta$, and consequently there is some $\alpha' < \kappa^*$ such that for every $\beta > \alpha'$, $c_\beta \not \in \bigcup B$. This implies that $\bigcup B$ does not cover the whole space and that $\set{c_\beta}{\beta<\kappa^*}$ is still non-meager, and so $\mx{cov}(\mathcal M_\kappa) \ge \kappa^*$ in $V[G]$ and $\mx{non}(\mathcal M_\kappa) \le \kappa^*$ in $V[G]$.

To finish the argument, we use the following inequalities which are provable in $\ZFC$ (see \cite{F:u} and \cite{Bat:gen} for details):

\begin{equation}\label{eq4}
\mx{add}(\mathcal M_\kappa) = \mx{min}\{\mf{b}(\kappa),\mx{cov}(\mathcal M_\kappa)\}, \mx{cof}(\mathcal M_\kappa) = \mx{max}\{\mf{d}(\kappa),\mx{non}(\mathcal M_\kappa)\},
\end{equation}
\begin{equation}\label{eq5}
\mf{b}(\kappa) \le \mx{non}(\mathcal M_\kappa), \mx{cov}(\mathcal M_\kappa) \le \mf{d}(\kappa).
\end{equation}

Then $\kappa^* = \mx{cov}(\mathcal M_\kappa)
= \mx{add}(\mathcal M_\kappa) = \mx{non}(\mathcal M_\kappa) = \mx{cof}(\mathcal M_\kappa)$ holds in $V[G]$ by (\ref{eq4}), (\ref{eq5}), and $\mf{b}(\kappa)= \mf{d}(\kappa) = \kappa^*$.
\end{proof}

Let us now discuss the values of $\mf{p}(\kappa)$ and $\mf{t}(\kappa)$. Since $\mf{p}(\kappa) \le \mf{t}(\kappa)$ follows easily, we will show $\mf{p}(\kappa) = \mf{t}(\kappa) =  \kappa^+$ by arguing in Lemma \ref{lm:tower} that $R$ adds a tower of size $\kappa^+$ (if $\kappa^* > \kappa^+)$. In preparation for Lemma \ref{lm:tower} let us state Fact \ref{f:SS} and prove Lemma \ref{lm:embed} which gives a sufficient condition for adding a tower of size $\kappa^+$.

Recall the definition of a tower in Definition \ref{def:tower} and of the tower number $\mf{t}(\kappa)$ in Definition \ref{def:t}. Recall that we write $A \sub^* B$ for $A,B \in [\kappa]^\kappa$ if $|A \setminus B| < \kappa$.

We will use the following special case of \cite[Main Lemma 2.1]{SS:bt}:

\begin{Fact}\label{f:SS}
Suppose $\kappa = \kappa^{<\kappa}$, $\kappa> \beth_\omega$, and $\mf{t}(\kappa) > \kappa^+$. Then there is an injective map $\varphi: 2^{<\kappa^+} \to [\kappa]^\kappa$ such that for each $\alpha < \kappa^+$ and $f \in 2^{\alpha}$, $\set{\varphi(\restr{f}{\beta})}{\beta < \alpha}$ is reversely well-ordered by $\sub^*$, satisfies SIP, and \begin{equation} \label{split} \varphi(f \conc 0) \cap \varphi(f \conc 1) = \emptyset.\end{equation}
\end{Fact}

Fact \ref{f:SS} gives a sufficient condition for a $\kappa^+$-distributive forcing to add a tower of size $\kappa^+$.

\begin{lemma}\label{lm:embed}
Suppose $\kappa^{<\kappa} = \kappa$, $\kappa > \beth_\omega$ and $\mf{t}(\kappa)>\kappa^+$. Suppose $P$ is a $\kappa^+$-distributive forcing which adds a new cofinal branch to the tree $(2^{<\kappa^+},\sub)$.\footnote{This is the same as saying that $P$ adds a fresh subset of $\kappa^+$, or that it fails to have the $\kappa^+$-approximation property.} Then $P$ adds a tower of size $\kappa^+$ and thus forces $\mf{t}(\kappa) = \kappa^+$.
\end{lemma}

\begin{proof}
In $V$, let $\varphi$ be the mapping from Fact \ref{f:SS}. Suppose $b \in 2^{\kappa^+}$ is a new cofinal branch through $(2^{<\kappa^+},\sub)^V$ in $V[P]$. Since $P$ is $\kappa^+$-distributive, $([\kappa]^\kappa)^V = ([\kappa]^\kappa)^{V[P]}$. Let $\mathcal T_b =  \set{\varphi(\restr{b}{\alpha})}{\alpha<\kappa^+}$. By the properties of $\varphi$, $\mathcal T_b$ is reversely well-ordered by $\sub^*$ and satisfies SIP. Moreover, $\mathcal T_b$ is a tower because it cannot have a pseudo-intersection in $V[P]$: If $A$ is a pseudo-intersection of $\mathcal T_b$, then $\set{\varphi^{-1}(B)}{A \sub^* B, B \in \rng{\varphi}}$ defines the cofinal branch $b$ in the ground model due to (\ref{split}). This is a contradiction, and so such an $A$ cannot exist.
\end{proof}

Note that the assumption $\kappa > \beth_\omega$ for an uncountable $\kappa$ in Fact \ref{f:SS} can be replaced by weaker conditions (see \cite[Main Lemma 2.1]{SS:bt}), and that the argument in Fact \ref{f:SS} and Lemma \ref{lm:embed} also works for $\omega = \kappa$ without additional assumptions.

\brm
Lemma \ref{lm:embed} provides a way of showing $\mf{t}(\kappa)= \kappa^+$ in $V[P]$ for a variety of forcing notions $P$ related to trees of height $\kappa^+$: the tree $(2^{<\kappa^+},\sub)$ embeds many trees of height $\kappa^+$, and adding a new cofinal branch to any of them by a $\kappa^+$-distributive forcing results in $\mf{t}(\kappa)=\kappa^+$. An example of this argument is in Lemma \ref{lm:tower} below.
\erm

Recall we work in the extension $V[F_0][G_\delta]$,  where $F_0$ and $G_\delta$ are generics fixed in the paragraph before Lemma \ref{lm:mutual}. Recall that $R = \dot{R}^{F_0}$, defined also in the paragraph before Lemma \ref{lm:mutual}, is a $\kappa^+$-distributive forcing in $V[F_0]$ and also in $V[F_0][G_\delta]$ by Lemma \ref{lm:mutual}. We use Lemma \ref{lm:embed} to argue that $R$ adds a tower of size $\kappa^+$ over $V[F_0][G_\delta]$ (if $\kappa^* > \kappa^+)$: 

\begin{lemma}\label{lm:tower}
The forcing $R$ adds a tower of size $\kappa^+$ over $V[F_0][G_\delta]$ if $\kappa^*>\kappa^+$. It follows $V[G]$ satisfies:
\begin{equation}
\kappa^+ = \mf{p}(\kappa) = \mf{t}(\kappa).
\end{equation}
\end{lemma}

\begin{proof} Since $\mf{p}(\kappa) \le \mf{t}(\kappa)$ is always true, it suffices to prove $\mf{t}(\kappa) = \kappa^+$. For $\alpha < \lambda$, let $F_0|\alpha$ denote the restriction of $F_0$ to $\Add(\kappa,\alpha)$. By Definition \ref{def:M} of $\M(\kappa,\lambda)$,\footnote{\label{ft:M} It is easy to check directly that the mapping which sends $(p,q) \in \Add(\kappa,\alpha) * \dot{\Add}(\kappa^+,1)$ to $(p,q')$ where $q'$ is a function with domain $\{\alpha\}$ and $q'(\alpha) = q$ is a complete embedding.} there is a complete  embedding $i_\alpha$: $$i_\alpha: \Add(\kappa,\alpha) * \dot{\Add}(\kappa^+,1) \to \M(\kappa,\lambda)$$ for every $\alpha \in \mathcal A$.  Fix any such $\alpha$. Let $T$ denote the tree $(2^{<\kappa^+},\sub)$ in $V[F_0|\alpha]$ of height $\kappa^+$. Since $\Add(\kappa,\alpha)$ forces the tail of the iteration $\Add(\kappa,[\alpha, \lambda)) * \dot{\P}_\delta \da p_{\dot{U}}$ to be $\kappa^+$-Knaster, $T$ has the same cofinal branches in $V[F_0|\alpha]$ as it has in $V[F_0][G_\delta]$.

Let us now work in $V[F_0][G_\delta]$. If $\kappa^* = \kappa^+$, then $\mf{t}(\kappa) = \kappa^+$, and there is nothing to prove because $V[G]$ will still satisfy $\mf{t}(\kappa) = \kappa^+$. So assume $\mf{t}(\kappa) = \kappa^*> \kappa^+$. Let $\varphi$ be the mapping from Fact \ref{f:SS} from $\tilde{T} = (2^{<\kappa^+})^{V[F_0][G_\delta]}$ to $([\kappa]^\kappa)^{V[F_0][G_\delta]}$. Note that $T$ is a subtree of $\tilde{T}$.

Due to the existence of $i_\alpha$, the generic filter $F_1$ for the $\kappa^+$-distributive quotient forcing $R$ adds over $V[F_0][G_\delta]$ a new cofinal branch through $T$ determined by the generic filter for the forcing $\Add(\kappa^+,1)^{V[F_0|\alpha]}$ (which is induced by $F_1$ and which adds a new cofinal branch to $T$). Now the result follows by Lemma \ref{lm:embed} applied in $V[F_0][G_\delta]$ with $R = P$, using the fact that a new cofinal branch through $T$ is a new cofinal branch through $\tilde{T}$ as well.
\end{proof}

This ends the proof of Theorem \ref{th:P_2}.
\end{proof}

\section{Open questions and further results}\label{sec:open}

We end the article with some open questions and further results.

\medskip

\textbf{Question 1.} 
The indestructibility results for the tree property at $\kappa^{++}$ in Fact \ref{f:tp} and for the negation of the weak Kurepa Hypothesis at $\kappa^+$ in Theorem \ref{th:wKH} only work for the iteration $\P_2$ in Theorem \ref{th:P_2}. At the moment we do not know whether it is consistent for an inaccessible $\kappa$ to have:

\begin{equation}\label{eq:o1} \kappa^+ < \mf{t}(\kappa) \le \mf{u}(\kappa) < 2^\kappa \mbox{ with } \TP(\kappa^{++}) \mbox{ and/or } \neg \wKH(\kappa^{+}).\end{equation}

A similar limitation of this indestructibility method appears in \cite{HS:u} where it is left open whether it is consistent to have a strong limit $\aleph_\omega$ with 

\begin{equation}\label{eq:o2} 2^{\aleph_\omega}>\aleph_{\omega+1}, \mf{u}(\aleph_\omega) = \aleph_{\omega+1} \mbox{ and }\TP(\aleph_{\omega+2}).\end{equation} 

Sometimes an ad hoc argument can be found in these contexts (as in \cite{FHS2:large} where (\ref{eq:o2}) is obtained without $\mf{u}(\kappa) = \aleph_{\omega+1}$), but technical difficulties are usually substantial and increase with the complexity of the forcing (in our case $\P_\delta\da p_{\dot{U}}$).

An underlying open question -- and arguably more interesting -- is therefore whether Fact \ref{f:tp} and/or Theorem \ref{th:wKH} can be extended to include all $\kappa^+$-cc or at least $\kappa^+$-Knaster forcings in $V[\M(\kappa,\lambda)]$.

\medskip

\textbf{Question 2.} (The situation at $\omega$.) Let us briefly sketch an argument that our results also apply to $\omega$: By \cite{BS:smallu}, one can force the consistency of $\omega_1 \le \mf{u} = \nu < \mf{d}=\delta = 2^\omega$ for arbitrary regular uncountable cardinals $\nu < \delta$ over a ground model with $\GCH$. The forcing is of the form $\Add(\omega,\delta)  * \dot{\Q}(\nu)$, where $\dot{\Q}(\nu)$ is an iteration of the Mathias forcing of length $\nu$. This forcing is $\omega_1$-Knaster, so mimicking the argument in Theorem \ref{th:P_2}, using Lemma \ref{lm:embed} with $\kappa = \omega$, we obtain:

\begin{theorem}\label{th:P_3}
Suppose $\GCH$ holds, $\lambda$ is a weakly compact cardinal and $\mu\ge\lambda$ and $\omega_1 \le \kappa^* < \mu$ are regular cardinals. Then $$\P_3:= \Add(\kappa,\lambda) * [ ( \Add(\kappa,\mu) * \dot{\Q}(\kappa^*)) \x \dot{R}]$$ forces $$\omega_1 = \mf{t} \le \mf{u} = \kappa^* < 2^\omega = \mu$$ and the compactness principles $$\SR(\omega_2), \TP(\omega_2), \neg \wKH(\omega_1) \mbox{ and }\DSS(\omega_2).$$
\end{theorem}

Notice that for $\omega = \kappa$ we obtain also the principle $\DSS(\omega_2$) and hence $\neg \AP(\omega_2)$ in the final model: It is known that $\M(\omega,\lambda)$ forces $\DSS(\omega_2)$ (see \cite{KR:DSS}), which is then preserved by the rest of the forcing $\P_3$, i.e.\ $\Add(\kappa,\mu) * \dot{\Q}(\kappa^*)$, by Theorem \ref{th:dss}. 

See Question 3 which is related to the principle $\DSS(\kappa^{++})$ for $\kappa >\omega$.

It is possible to compute the values of other cardinal invariants in this model as well. As in Question 1, we do not know how to construct a model in which all these compactness principles hold and yet $\omega_1 < \mf{t} = \mf{u} < 2^\omega$ holds as well.

\medskip

\textbf{Question 3.} We do not know whether Theorems \ref{th:P_1} or \ref{th:P_2} can be proved with  $\DSS(\kappa^{++})$ (and hence with $\neg \AP(\kappa^{++})$). See Remarks \ref{rm:new1} and \ref{rm:new2} for more context.

\medskip

\textbf{Question 4.} There are other compactness principles which may hold at $\kappa^{++}$ with suitably modified Theorems \ref{th:P_1} and \ref{th:P_2}. 

\medskip

\textbf{(a)} The club stationary reflection at $\nu^+$ for a regular $\nu$, $\CSR(\nu^+)$, states that for every stationary set $S$ in $\nu^+$ which contains only ordinals of cofinality $<\nu$, there is a club $C$ in $\nu^+$ such that $C$ intersected with ordinals of cofinality $\nu$ is contained in the set of reflection points of $S$. $\CSR(\nu^+)$ can be obtained by an iteration which shoots clubs through the sets of reflection points of all stationary sets with ordinals of cofinality $<\nu$, see \cite{m:sr} for more details. Honzik and Stejskalova showed in \cite{HS:u} a limited indestructibility result for $\CSR$: if $\kappa$ is regular, then Cohen forcing at $\kappa$ of arbitrary length preserves $\CSR(\kappa^{++})$, and so does the simple Prikry forcing at $\kappa$ if $\kappa$ is measurable. Recently, this has been extended to $\kappa^+$-linked forcings in \cite{TS:CSR} by Gilton and Stejskalova, but this is still not good enough for the application in this article.

\textbf{Remark.} $\CSR(\lambda)$ does not hold in $V[\M(\kappa,\lambda)]$, so in order to have $\CSR(\lambda)$ a (variant) of the iteration from \cite{m:sr} should be considered. Also note that $\CSR(\lambda)$ is compatible with $2^\kappa = \kappa^+$, and this is the setup of \cite{m:sr}. A generalization of the method in \cite{m:sr} for the Mitchell forcing and $2^\kappa = \kappa^{++} = \lambda$ appeared in \cite{SSGL:ind}.

\medskip

\textbf{(b)} The Guessing model principle at $\kappa^{++}$, denoted $\GMP(\kappa^{++})$, is a strong principle which implies $\neg \wKH(\kappa^+)$ and $\TP(\kappa^{++})$. $\GMP(\omega_2)$ is a consequence of $\PFA$, and $\GMP(\kappa^{++})$ holds in the Mitchell model $V[\M(\kappa,\lambda)]$ if $\lambda$ is supercompact in the ground model. Honzik, Lambie-Hanson, and Stejskalova proved in \cite{HLS:gmp} that $\GMP(\kappa^{++})$ is preserved over any model of $\GMP(\kappa^{++})$ by adding any number of Cohen subsets of $\kappa$, and its consequence $\neg \wKH(\kappa^+)$ is preserved over any model of $\GMP(\kappa^{++})$ by any $\kappa$-centered forcing. Preservation uder the $\kappa$-centered forcing notions is still not good enough for the present applications, but maybe the indestructibility result from \cite{HLS:gmp} can be strengthened accordingly to have $\GMP(\kappa^{++})$ in Theorems \ref{th:P_1} or \ref{th:P_2} (with $\lambda$ now being supercompact in the ground model).

\bibliography{mybiblio}
\bibliographystyle{amsplain}
\end{document}